\documentclass[a4paper,reqno,11pt]{amsart}

\usepackage{mathrsfs,bbm,amssymb,soul}
\usepackage{enumerate}
\usepackage{eufrak}
\usepackage{color,hyperref,framed}
\usepackage[margin=1.5cm, bottom=2cm, marginparwidth=1cm, marginparsep=2mm]{geometry}

\definecolor{darkblue}{rgb}{0.0,0.0,0.35}
\definecolor{darkgreen}{rgb}{0,0.35,0}
\definecolor{darkred}{rgb}{ 0.8,0,0}
\definecolor{lightgreen}{rgb}{0.75,1,0.75}
\hypersetup{colorlinks,breaklinks,linkcolor=darkblue,urlcolor=darkblue,anchorcolor=darkblue,citecolor=darkblue}

\def\fru{{\mathfrak{u}}}
\def\frB{{\mathfrak{B}}}
\def\gc{{g}}

\newtheorem{theorem}{Theorem}
\newtheorem{lemma}[theorem]{Lemma}

\theoremstyle{definition}
\newtheorem{definition}[theorem]{Definition}
\newtheorem{proposition}[theorem]{Proposition}
\newtheorem{remark}[theorem]{Remark}

\numberwithin{equation}{section}
\numberwithin{theorem}{section}

\newcommand{\rd}{{\rm d}}

\newcommand{\bq}{{\mathbf{x}}}

\newcommand{\bK}{{\mathbf{K}}}
\newcommand{\vJ}{{\mathbf{J}}}
\newcommand{\bE}{{\mathbf{E}}}

\newcommand{\bB}{{\mathbf{B}}}

\newcommand{\ba}{{\mathbf{a}}}
\newcommand{\bb}{{\mathbf{b}}}
\newcommand{\bg}{{\mathbf{g}}}
\newcommand{\bh}{{\mathbf{h}}}

\newcommand{\bu}{{\mathbf{u}}}
\newcommand{\bv}{{\mathbf{v}}}
\newcommand{\bx}{{\mathbf{x}}}
\newcommand{\by}{{\mathbf{y}}}
\newcommand{\bU}{{\mathbf{U}}}
\newcommand{\bV}{{\mathbf{V}}}
\newcommand{\bW}{{\mathbf{W}}}

\newcommand{\bX}{{\mathbf{X}}}
\newcommand{\bw}{{\mathbf{w}}}

\def\sfS{{\mathsf S}}

\def\cF{{\mathscr F}}
\def\Fo{\mr F}

\def\Eno{\mr{\overline{E}}}

\def\ep{\varepsilon}
\def\thx{\theta^{\ep}_{\bx}}

\def\thp{\theta^{\ep}_{\bv}}

\def\IC{\Big|_{t=0}}

\def\mo{^{\langle\ep\rangle}}
\def\vu{\mathbf{U}}	
\def\unlm{{\widetilde{\vu}\mo}} 		
\def\un{\overline{\vu}} 		
\def\unl{\widetilde\vu}	
\def\ue{\vu_\ep}	
\def\uee{\vu_{\ep}\mo}
\def\ueen{\vu_{\ep_n}^{\langle\ep_n\rangle}}
\def\vB{\mathbf{B}} 	
\def\Bnlm{{\widetilde{\vB}\mo}} 	
\def\Bn{\overline{\vB}}		
\def\Bnl{\widetilde\vB}
\def\Be{\vB_\ep}	
\def\Bee{\vB_{\ep}\mo}
\def\Been{\vB_{\ep_n}^{\langle\ep_n\rangle}}
\def\fn{\overline{f}}					 	 			
	
\def\fe{f_\ep}
\def\fb{f }	
\def\ub{\vu}	
\def\Bb{\vB}
\def\fen{f_{\ep_n}}	
\def\uen{\vu_{\ep_n}}		
\def\Ben{\vB_{\ep_n}}
\def\utt{\unlm}		
\def\Btt{\Bnlm}
\def\mR{{\mathbb{R}}}
\def\mT{{\mathbb{T}}}	
\def\mZ{{\mathbb{Z}}}
\def\mQ{{\mathbb{Q}}}
\def\mRp{{\mR^3}}	
\def\mTx{{\mT^3}}
\def\bC{{\mathcal C}}
\def\Cinf{\bC^\infty}

\def\cP{{\mathcal P}}
\def\cR{{\mathcal R}}
\def\cE{{\mathcal E}}

\def\mU{{\mathbbm U}}
\def\sA{{\mathsf A}}

\def\bW{{\mathbf W}}
\def\sK{{\mathsf K}}

\def\Rn{R}

\def\vv{\mathbf{V}}

\def\bX{{\mathbf X}}
\def\xx{\mathbf{X}}
\def\pp{\mathbf{V}}

\def\td{\textnormal{\,d}}
\def\ccdot{\cdot}

\def\lesssim{\stackrel{{}_<}{{}_\sim}}
\def\nc{\nabla_{\bx} \ccdot}
\def\ncx{\nabla_{\bx} \ccdot}
\def\cn{\ccdot\nabla_{\bx}}
\def\cnp{\ccdot\nabla_{\bv}}
\def\nt{\nabla_{\bx}\times}
\def\pa{\partial}
\def\pt{\pa_t}

\def\sss{\scriptscriptstyle}
\def\N{{\sss N}}

\def\dint{\displaystyle \int }
\def\ddt{{\td\over \td t}}

\def\Ettinc{{{\mathcal E}_\textnormal{total}^\textnormal{inc}}}

\newcommand{\Epar}[1]{\cE_\textnormal{par}[#1]}

\def\mr{\mathring}

\newcommand{\CCm}[1]{$#1 \in \bC([0,T]; \bC^m(\mTx;\mR^6))$ \;for all $m\ge0$}

\def\be{\begin{equation} }
\def\ee{\end{equation}}
\def\bpm{\begin{pmatrix}}
\def\epm{\end{pmatrix}}

\newcommand{\mychi}[1]{\chi_{(#1)}}
\newcommand{\intp}[1]{\int_{\mRp} \!#1\!\td^3\bv}
\newcommand{\pair}[1]{\langle #1 \rangle}
\usepackage{ulem}	\renewcommand{\emph}[1]{\textit{#1}}

\newcommand{\rem}[1]{}
\def\pheq{\phantom{=}}
\def\incL{{L^{2,\textnormal{inc}}}}
\def\scP{{\mathscr P}}

\title[Weak solutions to a hybrid Vlasov-MHD model in plasma physics]{Existence of Global Weak Solutions to a\\Hybrid Vlasov-MHD Model for Magnetized Plasmas}
\author[Bin Cheng]{Bin Cheng}
\address{Department of Mathematics, University of Surrey, Guildford, GU2 7XH, UK}
\email{b.cheng@surrey.ac.uk}

\author[Endre S\"uli]{Endre S\"uli}
\address{Mathematical Institute, University of Oxford, Woodstock Road, Oxford OX2 6GG, UK}
\email{endre.suli@maths.ox.ac.uk}

\author[Cesare Tronci]{Cesare Tronci}
\address{Department of Mathematics, University of Surrey, Guildford, GU2 7XH, UK}
\email{c.tronci@surrey.ac.uk}


\begin{document}

\begin{abstract}
We prove the global-in-time existence of large-data finite-energy weak solutions to an incompressible hybrid Vlasov-magnetohydrodynamic model in three space dimensions. The model couples three essential ingredients of magnetized plasmas: a transport equation for the probability density function, which models energetic rarefied particles of one species; the incompressible Navier--Stokes system for the bulk fluid; and a parabolic evolution equation, involving magnetic diffusivity, for the magnetic field.  The physical derivation of our model is given. It is also shown that the weak solution, whose existence is established, has nonincreasing total energy, and that it satisfies a number of physically relevant properties, including conservation of the total momentum, conservation of the total mass, and nonnegativity of the probability density function for the energetic particles. The proof is based on a one-level approximation scheme, which is carefully devised to avoid increase of the total energy for the
sequence of approximating solutions, in conjunction with  a weak compactness argument for the sequence of approximating solutions. The key technical challenges in the analysis of the mathematical model are the nondissipative nature of the Vlasov-type particle equation and passage to the weak limits in the multilinear coupling terms.
\end{abstract}

\maketitle

\section{\bf Introduction}\label{sec:intro}

In multiscale dynamics, hybrid kinetic-magnetohydrodynamic (MHD) theory offers the opportunity of a multi-physics modeling approach in which a macroscopic fluid flow is coupled to a kinetic equation incorporating the microscopic dynamics of a particle ensemble. Over the past decades, various hybrid models were formulated for different purposes, ranging from combustion theory \cite{Williams}  to polymeric fluid flows (see, for example, \cite{Constantin}, \cite{Barrett-Suli} and the references therein).

In plasma physics, linear hybrid schemes have now been used over several decades to model the interaction of the MHD bulk fluid with a rarified ensemble of energetic  particles, which in turn affect the overall energy and momentum balance. For example, in tokamak devices, the fusion reactions produce energetic rarefied alpha particles that escape a fluid description and thus require a kinetic treatment. Although the linear theory of hybrid kinetic-MHD schemes has been consolidated by computer simulations and analytical stability results \cite{Cheng2}, its nonlinear counterpart poses several consistency questions, which have only been approached during the last few years \cite{Tronci,TrTaCaMo,HoTr2011}. In particular, while different hybrid schemes are currently used in computer simulations, many of them have been found to lack energy conservation \cite{TrTaCaMo}, thereby generating unphysical instabilities above a certain frequency range. The formulation and analysis of hybrid kinetic-MHD models therefore 
represents a fascinating  research area, in which the analysis of nonlinear features of the  kinetic-MHD coupling necessitates the use of powerful modern mathematical techniques.

\subsection{\bf Mathematical setup}\label{sec:notation}

The physical derivation of the system of partial differential equations (PDE) studied in this paper is postponed to Section \ref{ss:lit}. In its original form, the system is stated in \eqref{rcc-hybrid} which is a model from the so-called current-coupling scheme (CCS), and is an incompressible, dissipative version of  \cite[eqs. (52)--(55)]{Tronci}. 
Here, for simplicity of the exposition, we shall set all positive physical constants appearing in  \eqref{rcc-hybrid}, including the density, to unity, as their specific values do not affect our considerations.

Suppose that $T>0$ and $\mT=\mR/(2\pi\mZ)$. For $t \in [0,T]$, $\bx \in \mTx$ and $\bv\in \mRp$, we shall seek the  3-vector velocity $\vu=\vu(t,\bx)$ of the bulk fluid, the  3-vector magnetic field $\vB=\vB(t,\bx)$, and the scalar probability density function $f(t,\bx,\bv)\ge 0$, which models energetic rarefied  particles of one species.
It will be implicitly understood throughout the paper that all functions of the variable $\bx \in \mTx$ satisfy $2\pi$-periodic boundary conditions with respect to $\bx$, and this property will only be explicitly stated when it is necessary to emphasize it.

The unknown functions $\vu$,
$\vB$ and $f$ are then required to satisfy
the following coupled system of nonlinear PDEs:
\begin{alignat*}{2}
&\pt f+\bv\cdot \nabla_{\bx} f= (( \vu-\bv)\times\vB)\cnp f
&&\quad ~\\
&\pt\vu+(\vu\cdot\nabla_{\bx})\vu+\nabla_{\bx} \mathcal{P}-(\nabla_{\bx} \times\vB)\times\vB- \Delta_{\bx}\vu=\int_{\mRp}{(\vu- \bv)\times\vB f}\td^3\bv
&&\quad \text{(subject to }\nabla_{\bx} \cdot\vu=0), \\
&\pt \vB- \nt(\vu\times\vB)= \Delta_{\bx}\vB&&\quad \text{(subject to }\nc\vB=0),
\end{alignat*}
where the auxiliary variable $\mathcal{P}$ denotes the ``pressure''.
Since $\vu$ and $\vB$ are divergence-free, one can apply the identities \eqref{A7} and \eqref{A6} in Appendix \ref{app:prod} to rewrite
\[(\nt\vB)\times\vB=(\vB\cdot \nabla_{\bx})\vB-{1\over2}\nabla_{\bx}|\vB|^2\,,\qquad \nt(\vu\times\vB)=(\vB\cdot\nabla_{\bx})\vu-(\vu\cdot\nabla_{\bx})\vB. \]
Then, by adding ${1\over2}|B|^2$ to the pressure $\mathcal{P}$, the incompressible Vlasov-MHD system can be restated as follows:
\vspace{-4mm}
\begin{subequations}
\label{eq}
\begin{alignat}{2}
&\label{f:eq}\pt f+\bv\cdot \nabla_{\bx} f= (( \vu-\bv)\times\vB)\cnp f
\rem{\\nonumber& =\nabla_\bv\cdot(( \vu-\bv)\times\vB f)},  &\qquad ~\\
\label{u:eq}&\pt\vu+(\vu\cdot \nabla_{\bx}) \vu-(\vB\cdot\nabla_{\bx})\vB- \Delta_{\bx}\vu=\int_{\mRp}{(\vu-\bv)\times\vB f}\td^3\bv -\nabla_{\bx} \mathcal{P}
&&\qquad (\text{subject to }\nc\vu=0), \\
\label{B:eq}&\pt \vB+(\vu\cdot\nabla_{\bx})\vB-(\vB\cdot\nabla_{\bx})\vu - \Delta_{\bx}\vB=0 &&\qquad (\text{subject to }\nc\vB=0).
\end{alignat}
\end{subequations}
It is this set of PDEs that we shall focus on on this article, subject to the initial conditions
\begin{equation}\label{initial}
\vu(0,\bx) = \mr\vu(\bx), \quad \vB(0,\bx) =\mr \vB (\bx),\quad f(0,\bx,\bv) =\mr f (\bx,\bv) \geq 0,\qquad \bx \in \mTx,\;
\bv \in \mRp,
\end{equation}
together with \textit{$2\pi$-periodic boundary conditions} with respect to $\bx$. 
The initial data  $\mr\vu(\bx)$ and $\mr\vB(\bx) $ are assumed to be {\it divergence-free} and $2\pi$-periodic.

We remark that $\nc\vB=0$ is an invariant of the evolutionary PDE system \eqref{eq}, at least when $\vu$ and $\vB$ have sufficiently many derivatives. Indeed, by taking the divergence of \eqref{B:eq}  and using the identity $   \sum_{j=1}^3(\pa_{x_j}\vu\cdot\nabla_{\bx})B_j-(\pa_{x_j}\vB\cdot\nabla_{\bx})U_j\equiv0$, we have that
\[\pt(\nc\vB)+(\vu\cn)(\nc\vB)-(\vB\cn)(\nc\vu)-\Delta_{\bx}(\nc\vB)=0,\]
so that if $\nc\vu=0$ for all times and the initial magnetic field $\mr\vB$ is divergence free, then $\nc\vB=0$ for all times.  A pressure-like term in the evolution equation for $\vB$ would be trivially constant and is therefore absent from \eqref{B:eq}. For weak solutions, which this article is concerned with, however, the divergence-free invariance of $\vB$ is not immediate; we therefore retain the divergence-free condition in \eqref{B:eq} for clarity. Another reason for explicitly stressing this condition is that our proof of the existence of weak solutions to the PDE system \eqref{eq} involves a sequence of approximate PDE systems to \eqref{eq}, which do not possess this divergence-free invariance. Thus, in these approximate systems, we must require $\vB$ to be divergence-free and add an explicit pressure-like term in the magnetic equation.

\subsection{Main result}\label{ss:result}
Here, we summarize our main result, Theorem \ref{thm:incomp}. 
It states the existence of global-in-time weak solutions to the PDE system \eqref{eq} in a sense to be made precise in Definition \ref{def:wk:inc}. This definition gives a weak formulation of \eqref{eq} along with certain physically relevant properties satisfied by such weak solutions.

Consider the {\it incompressible}, current-coupling scheme of the resistive Vlasov-MHD system \eqref{eq} for $t\in[0,\infty)$, with $\bq$
contained in a three-dimensional torus, $\mathbb{T}^3$, and $\bv$ contained in the whole three-dimensional space, $\mathbb{R}^3$. The given initial data are: the fluid velocity field $\vu(0,\bq)$, the initial magnetic field $\bB(0,\bq)$, which are both divergence-free and (Lebesgue) square-integrable, and the probability density of particles $f(0,\bq,\bv)$, which is pointwise nonnegative, (Lebesgue) integrable and essentially bounded. Suppose also that the initial energy is finite, i.e.,
\[\frac1{2} \int_{\mathbb{T}^3} \! \big(\big|  \vu(0,\bq) \big|^2+\big|  \bB(0,\bq) \big|^2\big)\rd^3\bx+\frac{1}{2}\int_{\mathbb{T}^3}\int_{\mathbb{R}^3}\!   f(0,\bq,\bv)  \left|\bv\right|^2\rd^3\bv\,\rd^3\bx<\infty.\]
Then, there exists a finite-energy global-in-time large-data weak solution $(f(t,\bq,\bv),\vu(t,\bq),  \bB(t,\bq))$ to the system \eqref{rcc-hybrid} for $t\in[0,\infty)$ in the sense of Definition \ref{def:wk:inc}. In particular, the total energy does not exceed its initial value  and the integrability properties assumed on the initial data as above are preserved in the course of evolution in time for all $t>0$.

\subsection{Mathematical literature}\label{ss:literature}

During the past decade several mathematical studies of PDE systems of coupled Navier--Stokes--Vlasov type (with or without the Fokker-Planck term) have been undertaken; the reader is referred, for example, to   \cite{MV2007,V:NS:drag,Yu:inc, Wang:Yu:comp}
and the references therein.  The existence of global weak solutions has been proved in these in several instances, using the key fact that the total energy is nonincreasing,
in conjunction with weak compactness arguments based on moment-estimates for the probability density function $f$. While these techniques have inspired the analysis performed in this paper, there is a significant difference in terms of the formulation: the coupling terms in the existing literature have almost exclusively taken to be of  (linear) drag-force type, which are proportional to $(\vu-\bv)$ or $\big(\vu- {\bK\over n}\big)$, whereas our model includes, instead, nonlinear coupling terms of Lorenz-force type, which is the natural choice from the point of view of plasma physics.

There are also a number of results concerning the existence of weak and classical solutions to Vlasov--Maxwell equations, without coupling to fluid dynamics. The list is long and we shall only mention \cite{DL:vlasov:maxwell} 
for the existence of global weak solutions, \cite{Gl:Sr:ARMA} for the existence of classical solutions {\it under the {\it a priori} assumption} that the plasma density vanishes for high velocities, and \cite[Chapter 5]{Glassey}, which includes a number of additional references on the subject. 
We note that the $\bv$-advection term of the Vlasov--Maxwell system is a constant times  $(\bv\times\bB+{ \bE})\cdot \nabla_\bv f$, so that when the ideal Ohm's law $\bE+\vu\times\bB=0 $ is used, it coincides with its counterpart in  \eqref{f:eq}.

The proof of the existence of global-in-time classical solutions to our hybrid kinetic-MHD system \eqref{eq}, which nonlinearly couples the incompressible Navier--Stokes equations to evolution equations for the magnetic field $\vB$
and the probability density function $f$, involves significant technical difficulties, even if one adopts an {\it a priori} assumption similar to that in \cite{Gl:Sr:ARMA}. Indeed, even for the source-free Navier--Stokes system the proof of the existence of global-in-time classical solutions for arbitrary smooth initial data is lacking, nor is there a counterexample to the breakdown of regularity of classical solutions. Our
study of the hybrid kinetic-MHD system \eqref{eq} therefore concentrates here on the existence of global-in-time weak solutions.

The rest of the article is organized as follows. 
In Section \ref{ss:lit}, we  present a detailed physical derivation of the system \eqref{eq} in the context of plasma physics and modeling. 
Then, we   formally prove   the key property that the total energy is nonincreasing in conjunction with rigorous proofs of moment-estimates in Section \ref{sec:cons}. We construct in Section \ref{sec:mo} a mollified system, using a single-level approximation,   for which   the desired regularity and energy bound are rigorously verified. The existence of a solution to this
approximating system is shown by proving it to be a fixed point of a mollified mapping. We carefully devise this mapping so as to ensure that its fixed point leads to a nonincreasing total energy. Finally in Section \ref{sec:compact}, we employ various compactness techniques to show that a subsequence of the sequence of approximating solutions converges weakly to a weak solution of the original PDE system. Since $f$ is governed by a transport equation without any diffusion, its regularity needs to be studied with particular care. We also address the lack of $L^1$ compactness, which is, to some extent, alleviated by the assumption that the initial datum for $f$ is in $L^r$ for all $r>1$. It is worth mentioning that the spatial $L^6$ integrability of $(\vu,\bB)$ and the trilinear coupling terms in the PDE require the moments $n,\bK$ to have rather high integrability indices, which we believe necessitates the high integrability indices of $f$.

\section{\bf  Plasma modeling and the physical derivation of the main PDE system}\label{ss:lit}

Many different hybrid models are available in the plasma physics literature \cite{PaBeFuTaStSu}, typically depending on whether the ``Hall term'' is retained in    Ohm's Law -- c.f. the comments above \eqref{ohmslaw1}. In the absence of a kinetic component, when the Hall term is neglected, the quasi-neutrality and the inertia-less electron assumptions lead to the most basic MHD fluid equations \cite{Freidberg}. Then, one may or may not consider resistivity effects, thereby obtaining a resistive or an ideal MHD model, respectively.

When the kinetic description of energetic particles is included,  the coupling of MHD to the kinetic component depends on the particular description that is adopted to model the energetic particles. Here, we shall focus on Vlasov-MHD models that neglect the Hall term, since such a treatment is customary in the nuclear fusion and solar physics literature. In this class of models, two main kinetic-MHD coupling schemes are discussed in the literature: the current-coupling scheme (CCS) \cite{BeDeCh,ParkEtAl,PaBeFuTaStSu,ToSaWaWaHo,Todo} and the pressure-coupling scheme (PCS) \cite{Cheng,ParkEtAl,PaBeFuTaStSu}, as they differ by the nature of the coupling terms in the fluid momentum equation. Upon adopting the Vlasov description for  energetic particle kinetics, references  \cite{Tronci,HoTr2011} showed that (in the ideal limit) the CCS conserves energy exactly as a consequence of its variational/Hamiltonian structures, while the PCS (as it appears in the literature) lacks an energy balance, unless extra inertial 
force terms are added to the Vlasov equation. These last terms are produced naturally by the variational/Hamiltonian approach and lead to an entirely new hybrid theory, which is currently under study and was shown to  reproduce Landau damping. The Lyapunov stability of energy-conserving hybrid models has recently been studied in \cite{TrTaMo1,TrTaMo2}, whereas previously proposed nonconservative models are known to exhibit unphysical instabilities \cite{TrTaCaMo}.

While the mathematical approach to kinetic-MHD theories is simplified by the use of the Vlasov equation for the kinetic component, practical computer simulations \cite{BeDeCh,Cheng,ParkEtAl,PaBeFuTaStSu,Todo,ToSaWaWaHo} employ the corresponding drift-kinetic or gyrokinetic approximations \cite{Brizard}: these are low-frequency kinetic equations that are obtained by sophisticated perturbation techniques to average out the fast Larmor gyration  around the magnetic field. While these low-frequency options are the subject of current research in terms of geometric variational methods \cite{BuTr2016}, here we shall consider the general case of full-orbit particle motion, thereby focusing on Vlasov-MHD models. As it was done in \cite{Tronci,HoTr2011,TrTaCaMo}, we shall modify the standard CCS appearing in the plasma physics literature \cite{ParkEtAl,ToSaWaWaHo,PaBeFuTaStSu} by replacing the drift-kinetic equation with the Vlasov equation. In its CCS variant, the set of partial differential equations (PDEs) of the 
Vlasov-MHD model (in the absence of collisional effects) reads in standard  plasma physics notation as follows:
\begin{subequations}\label{cc-hybrid}
\begin{align}
&\frac{\partial f}{\partial t}+\bv\cdot\nabla_{\bx} f+\frac{q_h}{m_h}\left(\bv-\vu\right)\times\bB\cdot\nabla_\bv f=0, \label{cc-hybrid-Vlasov}
\\
\label{cc-hybrid-momentum}
&\varrho\left(\frac{\partial\vu}{\partial
t}+\left(\vu\cdot\nabla_{\bx}\right)\vu\right)= \left(q_hn\, \vu-\,q_h\mathbf{K}+{\mu_0^{-1}\nabla_{\bx}\times\bB}\right)\times\bB
-\nabla_{\bx}{\cP},
\\
&\frac{\partial\varrho}{\partial
t}+\nabla_{\bx}\ccdot\left(\varrho\,\vu\right)=0,
\\
&\frac{\partial\bB}{\partial
t}=\nabla_{\bx}\times\big(\vu\times\bB\big),
\label{cc-hybrid-end}
\end{align}
\end{subequations}
where the operators $\nabla_{\bx}$ and $\nabla_{\bv}$ are understood to be taken with respect to the $\bx$ and $\bv$ variables,
respectively. The prognostic quantities are the probability density of the number of energetic particles, $f(t,\bq,\bv)$, of  dimension $(time)^3/(length)^6$; $\bB(t,\bq)$, denoting the magnetic field; and $\vu(t,\bq)$ and $\varrho(t,\bq)$, denoting the velocity and density of the bulk fluid, respectively.
The derived, diagnostic quantities (which are treated as auxiliary variables  throughout the article) are
$ n(t,\bq):=\int_{\mathbb{R}^3}\!f(t,\bq,\bv)\,\rd^3\bv$ of dimension $1/(length)^3$ denoting the total number of energetic particles per volume, and $ \mathbf{K}(t,\bq):=\int_{\mathbb{R}^3}\!\bv\,f(t,\bq,\bv)\,\rd^3\bv $    of dimension $1/\big[(length)^2\,(time)\big]$ denoting the sum of velocities of energetic particles per volume. Also,
$\mu_0$ denotes  the magnetic constant.
Note that the pressure ${\cP}$ is determined by an equation of state, which we shall assume to be barotropic, so that $\cP=\cP(\varrho)$, although the analysis in the current article is performed on an incompressible model, which is indifferent to the choice of the equation of state. Finally, the physical constants subscripted with $h$ (standing for ``hot'') are all associated with intrinsic properties of the  energetic particle species; in particular,
$q_h,\,m_h$ signify the charge and the mass of a single (energetic) particle, respectively, and  $a_h:={q_h}/{m_h}$ denotes the charge-to-mass ratio.

The total energy Hamiltonian,
\begin{equation} \label{Ham-preMHD}
H(f,\vu,\bB)=\frac1{2}\int_{\mathbb{R}^3} \!\varrho\left|\vu\right|^2\rd^3\bx+\frac{m_h}{2}\int_{\mathbb{R}^3}\int_{ \mathbb{R}^3}\! f \left|\bv\right|^2\rd^3\bv\,\rd^3\bx+\int_{\mathbb{R}^3}\!   \mathcal{U}(\varrho)\,\rd^3\bx
+\frac1{2\mu_0}\int_{\mathbb{R}^3} \left|\bB\right|^2\rd^3\bx,
\end{equation}
is conserved by the dynamics of \eqref{cc-hybrid}. Here, we have assumed a barotropic pressure law $\cP=\cP(\varrho)$, so that the internal energy per unit volume $\mathcal{U}$ depends only on the mass density $\varrho$ and satisfies $\varrho^2\,{\mathcal U}'(\varrho)=\cP(\varrho)$. Its
variational Euler--Poincar\'e and Hamiltonian structures     were   characterized in \cite{HoTr2011}, where conservations of the magnetic helicity $\int_{\mathbb{R}^3}\mathbf{A}\cdot\mathbf{B}\,\rd^3\bx$ and the cross helicity $\int_{\mathbb{R}^3}\vu\cdot\mathbf{B}\,\rd^3\bx$ were also verified explicitly.

The conservative properties of \eqref{cc-hybrid} are no longer true upon the introduction of collisional effects into the model. Such collisional effects are often incorporated in the plasma physics literature via a finite resistivity \cite{Todo} that breaks the so-called  ``frozen-in condition'' \eqref{cc-hybrid-end} (as it is expressed in terms of Lie-dragging, this condition enforces fluid particles on the same magnetic field line to always remain on the same field line).
In this paper, we adopt the same strategy in order to study the existence of weak solutions for a {\it resistive} variant of the system \eqref{cc-hybrid}. More particularly, we shall insert a finite resistivity in the problem so that the  total energy Hamiltonian \eqref{Ham-preMHD} decreases in time.  Although a complete physical treatment would also require incorporating the collisional effects emerging from the energetic particle dynamics, we shall look at a mathematically more tractable case here by only considering collisional effects in the MHD part of the model that are associated with the bulk fluid.

Now we derive the main focus of this article: a   member of the family of resistive Vlasov-MHD   models in its CCS variant. This model has appeared in the work of Belova and collaborators \cite{BelovaPark,BeGoFrTrCr}, who  implemented it in the {\it HYbrid and MHD simulation code} (HYM) to support fusion experiments.

More particularly, we shall focus on obtaining a consistent Ohm's law for the electric field $\bf E$.
The resistive term(s) shall be derived   via a standard procedure of adding collisional terms in the fluid momentum equations. We start with the full system of three sets of equations using the notations introduced between equations \eqref{cc-hybrid} and \eqref{Ham-preMHD}:
\begin{align}
\intertext{$\bullet$ kinetic Vlasov equation for energetic particles:}
\label{multi-fluid-Vlasov}
&\frac{\partial f}{\partial t}+\bv\cdot\nabla_\bx f +a_h\left(\bE+\bv\times\bB\right)\cdot\nabla_\bv f=0;
\\\intertext{$\bullet$ fluid equations for ions ($s=1$) and electrons ($s=2$) with momentum exchange via a friction term ${\mathbf R}$ as   the macroscopic description of collisional effects:}
\label{multi-fluid-momentum}
&\varrho_s\frac{\partial\vu_s}{\partial
t}+\varrho_s\left(\vu_s\cdot\nabla_{\bx}\right)\vu_s =
a_s\varrho_s\left(\bE+\vu_s\times\bB\right)-\nabla_{\bx}{\cP}_s+(-1)^s{\mathbf R},
\\\label{multi-fluid-mass}
&\frac{\partial\varrho_s}{\partial
t}+\nabla_{\bx}\ccdot \left(\varrho_s\vu_s\right)=0;\qquad\text{ and}
\\\intertext{$\bullet$  Maxwell equations:}
&\frac{\partial\bB}{\partial t}=-\nabla_{\bx}\times\bE,\qquad \text{(subject to $\nabla_{\bx}\ccdot \bB=0$)},
\label{faraday}
\\
&\mu_0\epsilon_0\frac{\partial\bE}{\partial
t}=\nabla_{\bx}\times\bB-\mu_0\sum_{s=1}^2
a_s\varrho_s\vu_s-\mu_0\,q_h\mathbf{K},\label{ampere:Max}
\\
&\epsilon_0 \nabla_{\bx}\ccdot \bE=\sum_{s=1}^2 a_s\varrho_s+q_hn;
\label{multi-fluid-gauge}
\end{align}
where the physical constants $q_s,m_s$ for $s=1,2$, just like their counterparts for energetic particles, denote the charge and the mass, respectively, of a single ion ($s=1$) or a single electron ($s=2$) and $a_s:=q_s/m_s$ denotes the charge-to-mass ratio. By the nature of all relevant physical settings, we have $m_h$ and $m_1$ at the same scale and $m_2$ extremely small. Then, by the fact that the energetic particles are very rarefied, we have the scaling regime
\be\label{rare:scale}
m_1n/\varrho_1\ll 1.\ee
Since $q_h,q_1,q_2$ are all at the same scale, the above relation implies $q_hn\ll a_1\varrho_1$, which will be directly used later.

The opposite signs of the collisional/frictional term ${\mathbf R}$ in the two momentum equations ensure conservation of the total momentum. The detailed derivation of ${\mathbf R}$ starting from particle or kinetic description is beyond the scope of this article, and we only refer to \cite[(2.17)]{Callen}, \cite[(3.105)]{Hut2001:plasma} and state that it is of the form
$${\mathbf R}=\nu\varrho_2(\vu_1-\vu_2),\qquad \nu=C_{\mathbf R}\varrho_1,$$
where the positive parameter $\nu$ is the  Maxwellian-averaged electron-ion collision frequency (despite its name, $\nu$ is in fact the average momentum {\it relaxation} rate for the slowly changing Maxwellian distribution of {\it electrons}). Also, the positive parameter $C_{\mathbf R}$ can be well approximated by a constant for barotropic flows since both \cite[(2.17)]{Callen} and \cite[(3.105)]{Hut2001:plasma} show that the electron-ion collision frequency $\nu$ is proportional to the number of ions
per unit volume, which is, apparently, proportional to the ion density $\varrho_1$. The factor $(\vu_1-\vu_2)$ in the formula for ${\mathbf R}$ is also consistent with the fact that the (macroscopic) collisional effects are determined by the  (macroscopic) drift velocity of the electrons relative to the ions.

We will later insert kinematic viscosity as part of the standard Navier--Stokes equations, which accounts for collisions amongst ions. Here, we have made the assumption that ions and electrons do not collide with the  energetic particles, which are themselves  assumed to be also collisionless. Although this assumption may not be completely justified, we shall pursue this direction to simplify the problem as much as possible. Further, we perform the same approximations as in standard  MHD theory \cite{Freidberg}. First, the enormous disparity in masses $m_1\gg m_2$   allows us to approximate \eqref{multi-fluid-momentum} with $s=2$ by neglecting the left-hand side terms, resulting in
\begin{equation}
\mathbf{0}=
a_2\varrho_2\left(\bE+\vu_2\times\bB\right)-\nabla_{\bx}{\cP}_2+ C_{\mathbf R}\varrho_1\varrho_2(\vu_1-\vu_2).
\label{eleq1}
\end{equation}
Then, adding this to \eqref{multi-fluid-momentum} with $s=1$ produces
\begin{equation}\label{TotMom}
\varrho_1\frac{\partial\vu_1}{\partial
t}+\varrho_1\left(\vu_1\cdot\nabla_{\bx}\right)\vu_1 =
\sum_{s=1}^2
a_s\varrho_s\bE+\sum_{s=1}^2
a_s\varrho_s\vu_s\times\bB-\nabla_{\bx}({\cP}_1+{\cP}_2),
\end{equation}
where the collisional terms $\pm {\mathbf R}$ cancel out. Next, upon assuming quasi-neutrality by the formal limit $\epsilon_0\to0$ that is applied to the Maxwell's equations {\eqref{ampere:Max}}, \eqref{multi-fluid-gauge}, the electromagnetic fields satisfy the equations
\begin{align}\label{ampere}
\text{zero displacement current: }\qquad&
a_1\varrho_1\vu_1+a_2\varrho_2\vu_2+\,q_h\mathbf{K}=\frac1{\mu_0}\nabla_{\bx}\times\bB\quad(=\vJ)
,\\
\text{quasi-neutrality: }\qquad&
a_1\varrho_1+ a_2\varrho_2+q_hn=0
\label{gauss}
,
\end{align}
{where $\vJ:=\mu_0^{-1}\nabla_{\bx}\times\bB$ denotes the electric current density in the system, and is henceforth always an auxiliary, diagnostic variable.} Then, the ion momentum equation \eqref{TotMom} becomes (with ${\cP}={\cP}_1+{\cP}_2$)
\begin{equation}
\varrho_1\frac{\partial\vu_1}{\partial
t}+\varrho_1\left(\vu_1\cdot\nabla_{\bx}\right)\vu_1 = -q_hn\,\bE+\left(\vJ-\,q_h\mathbf{K}\right)\times\bB
- \nabla_{\bx}{\cP}
\label{CCSmomentum}
.
\end{equation}
Thus we have now reduced the two-fluid model  \eqref{multi-fluid-momentum}, \eqref{multi-fluid-mass} to a single-fluid model, for which we   retain the continuity equation only for $s=1$ as well. One can then combine these with the kinetic equation \eqref{multi-fluid-Vlasov}, Faraday's law \eqref{faraday} and use the elementary identities listed in Appendix \ref{app:prod}
to formally prove conservation of the total momentum:
\[\int_{\mathbb{R}^3}\!\Big(\varrho_1\vu_1+\int_{\mathbb{R}^3}\!\bv f\td^3\bv\Big)\td^3\bx\]
and can also formally deduce the rate of change of the total energy; indeed, by considering the Hamiltonian $H$ defined in \eqref{Ham-preMHD}, we have that
\be\label{ddt:H}\begin{aligned}\ddt H(f,\vu_1,\bB)&=-\int_{\mathbb{R}^3}\!\Big(\big(\vJ-q_h\bK+q_hn\,\vu_1)\ccdot\bE+\big( \vJ-q_h\bK\big)\ccdot\big(\vu_1\times\bB\big)\Big)\td^3\bx\\
    &=-\int_{\mathbb{R}^3}\!\Big(\big(\vJ-q_h\bK+q_hn\,\vu_1)\ccdot\big(\bE+ \vu_1\times\bB\big)\Big)\td^3\bx.
   \end{aligned}
\ee

Next, in order to relate the electric field $\bE$ to the prognostic unknowns and thus close the system, we combine   Amp\`ere's current balance \eqref{ampere} and quasi-neutrality \eqref{gauss} to obtain
\be
\vu_2=  \frac{-1}{a_1 \varrho_1+q_hn}\big( \vJ-q_h\mathbf{K} -a_1\varrho_1\vu_1\big),
\ee
so that, by simple manipulation,
\be\label{vu1:vu2}
\vu_1-\vu_2=   \frac{1}{a_1 \varrho_1+q_hn}\,\left(\vJ-{q_h}\mathbf{K}   +q_hn\,\vu_1\right).\ee

On the other hand, by the identity $\bE+\vu_1\times\bB= (\vu_1-\vu_2)\times\bB+(\bE+\vu_2\times\bB)$ and the inertia-less  electron momentum equation \eqref{eleq1},
\begin{align*}
\bE+\vu_1\times\bB
=&\
(\vu_1-\vu_2)\times\bB+\frac{1}{a_2}\varrho_2\nabla_{\bx}{\cP}_2-\frac{C_{\mathbf R}\varrho_1}{a_2}\left(\vu_1-\vu_2 \right),
\end{align*}
and therefore, upon substituting \eqref{vu1:vu2}, we have
\begin{align*}
&\bE+\vu_1\times\bB\\
&=
\frac{1}{a_1 \varrho_1+q_hn}\big(\vJ- q_h\mathbf{K} +q_hn\,\vu_1\big)\times\mathbf{B}
-\frac{1}{a_1\varrho_1+q_hn}\nabla_{\bx}{\cP}_2
-\frac{C_{\mathbf R}\varrho_1}{a_2  (a_1 \varrho_1+q_hn)}\left(\vJ-{q_h}\mathbf{K}   +q_hn\,\vu_1\right).
\end{align*}
Then, we imitate the derivation of ideal MHD  \cite{Freidberg} and assume that the ``Hall effect'' $\vJ\times\mathbf{B}$ and electron pressure gradient $\nabla_{\bx}{\cP}_2$ are both negligible compared to the  Lorentz force $a_1\varrho_1\vu_1\times\mathbf{B}$. This step leads to the relation
\begin{equation}\label{ohmslaw1}
\bE
+\vu_1\times\bB=\frac{q_h}{a_1 \varrho_1+q_hn}\big(n\,\vu_1- \mathbf{K} \big)\times\mathbf{B}-\frac{C_{\mathbf R}\varrho_1}{a_2  (a_1 \varrho_1+q_hn)}\left(\vJ+q_hn\,\vu_1-{q_h}\bK  \right).
\end{equation}
Then, thanks to the assumption \eqref{rare:scale} for energetic particles (so that $\frac{1}{a_1 \varrho_1+q_hn}\simeq\frac1{a_1\varrho_1}$) and the fact that $a_1,a_2$ have opposite signs, we are justified to consider the resistivity $\eta:= -C_{\mathbf R}/(a_2  a_1 )$ to be a positive constant.
Compared to $\vu_1\times\bB$ on the left-hand side of \eqref{ohmslaw1}, the term $\frac{q_h}{ a_1 \varrho_1+q_hn}\,(n\,\vu_1 )\times\mathbf{B}$ on the right-hand side is negligible because of \eqref{rare:scale}. Analogously, upon introducing the average particle velocity\footnote{This quantity is either very low or at most  comparable with the MHD fluid velocity $\vu_1$.  This is  consistent with the hypothesis of energetic particles,  since  the latter hypothesis involves the temperature rather than the mean velocity. Denoting the temperatures of the hot and fluid components by $T_h$ and $T_f$, respectively, we have   $T_h\gg T_f$ (see \cite{Cheng}).  With the  definition of the temperature  $T_h=\left(m_h/3nk_B\right){\int_{\mathbb{R}^3}|\mathbf{v}-\boldsymbol{W}|^{2} f\,{\rm d}^3\mathbf{v}}$ (where $k_B$ denotes Boltzmann's constant), the assumption on  the  energetic component amounts to an assumption on the trace of the second-order moment of the Vlasov density with no assumption on the mean velocity,
which is actually low for hot particles close to isotropic equilibria.}
$\boldsymbol{W}=\mathbf{K}/n\lesssim\vu_1$,  the term $\frac{q_hn}{ a_1 \varrho_1+q_hn}\big( - \boldsymbol{W} \big)\times\mathbf{B}$ is also seen to be negligible. Therefore, we are left with the following formula,  which we shall refer to as ``extended Ohm's law'' (c.f. \cite{Todo}):
\begin{equation}\label{ohmslaw2}
\bE+\vu_1\times\bB=\eta\left(\vJ+q_hn\,\vu_1-{q_h}\bK  \right).
\end{equation}
Although \eqref{ohmslaw2} consistently guarantees that the rate of change \eqref{ddt:H} for the total energy is non-positive, the complicated form of the extended Ohm's law  leads to significant difficulties in the mathematical
analysis of the model. Thus, we have simplified the problem by invoking, once again,
the assumption \eqref{rare:scale} for energetic particles
to obtain  the usual form of Ohm's law:
\be
\bE+\vu_1\times\bB=\eta\,\vJ.
\label{Ohmslaw}
\ee
However, a {\it consistency issue} emerges here: this approximation does not guarantee the nonpositivity of  the time rate  \eqref{ddt:H} for the total energy.
In order to progress further, we make one additional approximation:  we neglect   {\it all} resistive force terms  in the ion momentum equation and kinetic equation, namely we use the ideal Ohm's law   $\bE+\vu_1\times\bB=0$  in \eqref{CCSmomentum} and  \eqref{multi-fluid-Vlasov}, but we use the usual Ohm's law \eqref{Ohmslaw} in    Faraday's Law \eqref{faraday}. As has been noticed in \cite{BeJaJiYaKu}, this step is needed for momentum conservation and it amounts to defining an effective electric field given by ${\bf E}-\eta{\bf J}$, where $-\eta{\bf J}$ represents the collisional drag on the ions and the hot particles. Then, this results in an approximation of the kinetic equation \eqref{multi-fluid-Vlasov}, Faraday's law \eqref{faraday} and the ion momentum equation \eqref{CCSmomentum} by the following current-coupling scheme of resistive Vlasov-MHD (where $\vu=\vu_1$, and we also incorporate the kinetic viscosity $\kappa$,  incompressibility and the constant ion density $\bar{\varrho}$):
\begin{subequations}\label{rcc-hybrid}
\begin{align}
&\frac{\partial f}{\partial t}+\bv\cdot\nabla_\bx f +\frac{q_h}{m_h}\left(\bv-\vu\right)\times\bB\cdot\nabla_\bv f=0, \label{rcc-hybrid-Vlasov}
\\\label{rcc-hybrid-momentum}
&\bar{\varrho}\Big(\frac{\partial\vu}{\partial
t}+\left(\vu\cdot\nabla_{\bx}\right)\vu\big)= q_hn \,\vu\times\bB+\left(\vJ-\,q_h\mathbf{K} \right)\times\bB\nonumber\\
&\hspace{4cm} +\kappa\Delta_{\bx}\vu-\nabla_{\bx}{\cP}\,,\qquad\  \text{(subject to $\nabla_{\bx}\cdot\vu=0$)},
\\
& \frac{\partial\bB}{\partial
t}=\nabla_{\bx}\times\left(\vu\times\bB\right) +{\eta\over\mu_0}\Delta_{\bx}\bB
,\hspace{2.5cm}  \text{(subject to $\nabla_{\bx}\cdot\bB=0$)}.
\label{rcc-hybrid-end}
\end{align}
\end{subequations}
Here the unknowns are $\vu(t,\bq)$, $f(t,\bq,\bv)$, $\bB(t,\bq)$, and the auxiliary variables involved are $\cP$ and, as we have defined before,
$$\vJ=\mu_0^{-1}\nabla_{\bx}\times\bB,\quad
n(t,\bq)=\int_{\mathbb{R}^3}\!f(t,\bq,\bv)\,\rd^3\bv,\quad \mathbf{K}(t,\bq)=\int_{\mathbb{R}^3}\!\bv\,f(t,\bq,\bv)\,\rd^3\bv.$$
The symbols  $q_h,\kappa,m_h,\eta,\mu_0$ denote  positive physical constants.

This system of hybrid Vlasov-MHD equations is implemented in the HYM code, as has been recently presented in \cite{BeGoFrTrCr}. The remainder of this paper is devoted to an analytical study of its parameter-free version \eqref{eq}. This system exhibits conservation of total momentum and nonincreasing total energy, thanks to a calculation similar to the one leading to \eqref{ddt:H}, which will be discussed in Section \ref{sec:cons}. These physical properties, in fact, play a crucial role in our mathematical analysis of the model.

\section{\bf Conservation properties and bounds on the moments}\label{sec:cons}

In this section, assuming sufficient regularity of the solution to system  \eqref{eq}, and $2\pi$-periodic boundary conditions with respect to $\bx$ and suitably rapid decay of $f$ as $|\bv| \rightarrow \infty$, we shall present  \emph{formal} proofs of various balance laws and energy inequalities. Although subsequently we shall study the system \eqref{eq} only, corresponding to the incompressible case, it is instructive at this point to discuss the (formal) energy equality in the compressible case as well. The argument   in the incompressible case will be made rigorous later on in the paper by fixing the function spaces in which the unknown functions $f$, $\bU$ and $\bB$ are sought. The question of existence of a global weak solution to the {\it compressible} model will be studied elsewhere.

We also show in Proposition \ref{prop:Lr} that, at any time $t\ge0$, the $L^r(\mTx)$ norm (with a suitable values of $r$, whose choice will  be made clear below,) of the moments of $f$ are bounded in terms of the total energy and the $L^\infty(\mTx \times \mRp)$ norm of $f$, both of which will later be rigorously shown not to exceed their respective initial sizes.

The equation  \eqref{f:eq} is a transport equation with divergence-free ``velocity fields'' with respect to both the $\bx$  and $\bv$ coordinates. That is to say,
$$\nabla_{\bx}\ccdot \bv =\nabla_{\bv}\ccdot((\bv-\vu)\times\vB)=0.$$
As a consequence, the $L^r(\mTx \times \mRp)$ norm of $f$ is constant in time for all $r\in[1,\infty)$. In addition, one can show by the method of characteristics that the minimum and maximum values of $f$ are preserved in the course of temporal evolution, and therefore the $L^\infty(\mTx \times \mRp)$ norm of $f$ is also constant in time; consequently, for a nonnegative initial datum $\mr f$  the associated solution $f$ remains nonnegative in the course of evolution in time.

It is also straightforward to show (formally) the conservation of the total momentum
\[\int_{\mTx} \Big[\Big(\int_{\mRp}\bv f\td^3\bv\Big)+ \vu\Big] \td^3\bx\]
using the elementary identities from Appendix \ref{app:prod}.

There are three contributions to the total energy of the system: from energetic particles, from the kinetic  energy of the bulk fluid, and from the magnetic field. The total energy is therefore defined as follows:\footnote{ The electric field $\bE$ also stores energy, but with our scalings here its contribution is neglected. In order to justify this, we return to physical units and consider linear materials with homogeneous permittivity $\ep$ and permeability $\mu$, so that $\mathbf{D}=\ep \bE$ and $\mathbf{H}=\mathbf{B}/\mu$. In MHD models, Faraday's law $\pt\vB+\nt \bE=0$ is used, which implies the scaling law
\(
{[\vB]\over [t]}\sim {[\bE]\over [\bx]}.
\)
Meanwhile, the MHD approximation adopts the zero displacement-current limit of the Maxwell--Amp\'ere equation $\nt \mathbf{H}-\mathbf{J}= \pt \mathbf{D}\approx \mathbf{0}$, which implies another scaling law:
\(
{[\mathbf{H}] \over [\bx]}\gg { [\mathbf{D}]\over [t]}.
\)
Multiplying these two scaling laws we obtain
\[
[\mathbf{B}][\mathbf{H}]  \gg[\bE][\mathbf{D}].
\]
Therefore, the contribution of the electric field to the total electromagnetic energy density,  ${1\over2}\left(\bE\cdot \mathbf{D} + \mathbf{B}\cdot \mathbf{H} \right)$, is negligible.}
\[\Ettinc=\Ettinc[f,\vu,\vB]:=\int_{\mTx} \bigg[\bigg(\int_{\mRp}{1\over2} |\bv|^2 f\td^3\bv\bigg)+{1\over2} |\vu|^2 +{1\over2}|\vB|^2\bigg]\td^3\bx.
\]
Assuming that
basic physical laws are obeyed, we must have conservation of the total energy when all dissipation terms are set to zero.
In order to illustrate the energy budget and the energy exchange between the equations in the system, we introduce the following \emph{energy conversion rates}:
\begin{alignat*}{2}
&\cR_1:=\dint_{\mTx}\bigg[( \vu\times\vB)\ccdot{\int_{\mRp} \bv f\td^3\bv}\bigg]\td^3\bx&&\quad\text{(energy of the   particles to kinetic energy of the fluid);}\\
&\cR_2:=\dint_{\mTx}(\nt\vB)\ccdot( \vu\times\vB)\td^3\bx&&\quad\text{(kinetic   energy  of the   fluid to   magnetic energy).}\\
\end{alignat*}
We shall now decompose ${\ddt}\Ettinc $ into its three constituents in order to highlight the roles of these energy conversion rates.
\begin{enumerate}[(i)]
\item \label{en:exchange:1}
\textit{Change in the total energy of energetic particles.}
Integrating ${1\over2}|\bv|^2\ccdot\eqref{f:eq}$ over $\mTx\times\mRp$ and performing integration by parts yields
\[
\begin{aligned}
{\ddt}\int_{\mTx \times \mRp}{1\over2} |\bv|^2 f\td^3\bx\td^3\bv&=\int_{\mTx \times \mRp}{1\over2} |\bv|^2\, \nabla_{\bv}\ccdot((\vu-\bv)\times\vB f)\td^3\bx\td^3\bv\\
&=-\int_{\mTx \times \mRp} \bv \ccdot((\vu-\bv)\times\vB f)\td^3\bx\td^3\bv\\
&=- \int_{\mTx}\Big(\int_{\mRp} \bv f\td^3\bv\Big)\ccdot(\vu\times\vB)\td^3\bx\\
& =-\cR_1.
\end{aligned}
\]

\item \textit{Change in the kinetic energy of the bulk fluid.}
Integrating $\eqref{u:eq}\ccdot \vu$ over $\mTx$ and performing integration by parts yields
\[
\begin{aligned}
&\pheq{\ddt}\int_{\mTx} {1\over2} |\vu|^2\td^3\bx\\
&\quad= -\int_{\mTx}(\nt\vB)\ccdot( \vu\times\vB)\td^3\bx - \int_{\mTx} |\nabla_{\bx}\vu|^2+ (\nc\vu)^2 \td^3\bx+\int_{\mTx}( \vu\times\vB)\ccdot\Big({\int_{\mRp} \bv f\td^3\bv}\Big)\td^3\bx\\
&\quad=-\cR_2- \int_{\mTx} |\nabla_{\bx}\vu|^2\td^3\bx+\cR_1.
\end{aligned}
\]
\item \textit{Change in the magnetic energy.}
By integrating $\eqref{B:eq}\ccdot \vB$ over $\mTx$, and using the identity  $\nc(\vu'\times\vB)=(\nt\vu')\ccdot\vB-(\nt\vB)\ccdot\vu'$ with $\vu':=\vu\times\vB$ (cf. \eqref{A4}) we obtain, after integrating by parts, that
\[
{\ddt}\int_{\mTx} {1\over2}|\vB|^2\td^3\bx=\int_{\mTx}(\nt\vB)\ccdot(\vu\times\vB)\td^3\bx- \int_{\mTx} |\nabla_{\bx}\vB|^2\td^3\bx=\cR_2- \int_{\mTx}  |\nabla_{\bx}\vB|^2\td^3\bx.
\]
\end{enumerate}

To conclude, for a smooth solution $(f,\vu,\vB)$ to \eqref{eq},
with $f$ decaying sufficiently rapidly as $|\bv|\to\infty$,  we have that
\be
\label{ddt:Einc}
\Ettinc (t) = \Ettinc (0)-\int_0^t \int_{\mTx}\left( |\nabla_{\bx}\vu|^2+   |\nabla_{\bx}\vB|^2\right)\td^3\bx\td s\qquad \forall\,t \in [0,T],
\ee
as well as
\be
\label{f:Lr}\|f(t)\|_{L^r(\mTx \times \mRp)}=\|f_0\|_{L^r(\mTx \times \mRp)}\qquad \forall\, t \in [0,T], \quad \forall\, r \in[1,\infty].
\ee

These bounds on $f$ then imply the relevant bounds on the moments of $f$ in the following sense.
\begin{proposition}\label{prop:Lr}
Consider a measurable nonnegative function $f : (t,\bx,\bv)\in [0,T]\times \mTx \times \mRp \mapsto f(t,\bx,\bv)\in\mR$ such that  $\|f(t,\cdot,\cdot)\|_{L^\infty(\mTx \times \mRp)}<\infty$ for $t \in [0,T]$, and assume that
\be
\label{def:Ef}
\Epar f(t):=\int_{\mTx \times \mRp}{1\over2} |\bv|^2 f(t,\bx,\bv)\td^3\bx\td^3\bv < \infty\qquad
\mbox{for $t \in [0,T]$}.
\ee
Then, the following bounds on the zeroth, first and second moment of $f$ hold for $t \in [0,T]$:
\begin{align}
\label{bd:p0f}
&\Big\|\int_{\mRp} f(t,\cdot,\bv)\td^3\bv\Big\|_{L^{5\over3}(\mTx)}\le C\|f(t,\cdot,\cdot)\|_{L^\infty(\mTx \times \mRp)}^{2\over5}\,\Big(\Epar {f}(t)\Big)^{3\over5}; \\
\label{bd:pf}
&\Big\|\int_{\mRp} |\bv|\,f(t,\cdot,\bv)\td^3\bv\Big\|_{L^{5\over4}(\mTx)}\le C\|f(t,\cdot,\cdot)\|_{L^\infty(\mTx \times \mRp)}^{1\over5}\,\Big(\Epar {f}(t)\Big)^{4\over5};\\
\label{bd:ps}
&\Big\|\int_{\mRp} |\bv|^2 f(t,\cdot,\bv)\td^3\bv\Big\|_{L^{1}(\mTx)}\le C \Epar {f}(t).
\end{align}
More generally,  for any real number $k\in[0,2]$ and $t \in [0,T]$, we have that
\be
\label{bd:pkf}
\Big\|\int_{\mRp} |\bv|^k f(t,\cdot,\bv)\td^3\bv\Big\|_{L^{5\over3+k}(\mTx)}\le C\|f(t,\cdot,\cdot)\|_{L^\infty(\mTx \times \mRp)}^{2-k\over5}\,\Big(\Epar {f}(t)\Big)^{3+k\over5}.
\ee
\end{proposition}

We note that the bound \eqref{bd:p0f} is stronger than the $L^1(\mTx \times \mRp)$ integrability of $f$, and \eqref{bd:pf} is stronger than the result of applying H\"older's inequality to the product $\bv\sqrt f\,\sqrt f$ over $\mTx \times \mRp$.

\begin{proof}
Take any $N>0$ and let $C$ denote a generic positive constant, independent of $N$, whose value may vary from line to line. Then, with $(t,\bx) \in [0,T] \times \mTx$ fixed,
\be
\label{prop:inter}
\begin{aligned}0
\le\int_{\mRp} f(t,\bx,\bv)\td^3\bv&=\int_{|\bv|\le N} f(t,\bx,\bv)\td^3\bv+\int_{|\bv|> N}f(t,\bx,\bv)\td^3\bv\\
&\le CN^3\|f(t,\cdot,\cdot)\|_{L^\infty(\mTx \times \mRp)}+N^{-2}\int_{\mRp} |\bv|^2f(t,\bx,\bv)\td^3\bv.
\end{aligned}
\ee
Now, again with $(t,\bx) \in [0,T] \times \mTx$ fixed, the right-hand side in the last inequality attains its minimum at %
\[
N=N(t,\bx) = C\Big(
\int_{\mRp} |\bv|^2f(t,\bx,\bv)\td^3\bv\,/\,\|f(t,\cdot,\cdot)\|_{L^\infty(\mTx \times \mRp)}\Big)^{1\over5},
\]
and therefore
\[
0\le\int_{\mRp} f(t,\bx,\bv)\td^3\bv\le C\|f(t,\cdot,\cdot)\|_{L^\infty(\mTx \times \mRp)}^{2\over5}\Big(\int_{\mRp} |\bv|^2f(t,\bx,\bv)\td^3\bv\Big)^{3\over5}.
\]
Hence,
\begin{align*}
\Big\|\int_{\mRp} f(t,\cdot,\bv)\td^3\bv\Big\|^{5\over3}_{L^{5\over3}(\mTx)}&\le C\|f(t,\cdot,\cdot)\|_{L^\infty(\mTx \times \mRp)}^{2\over3}\int_{\mTx}\int_{\mRp}{1\over2} |\bv|^2 f(t,\bx,\bv)\td^3\bv\td^3\bx\\
&\le C\|f(t,\cdot,\cdot)\|_{L^\infty(\mTx \times \mRp)}^{2\over3}\, \Epar {f}(t),
\end{align*}
which directly implies \eqref{bd:p0f}.

The inequality \eqref{bd:pf} is a consequence of \eqref{bd:p0f} and H\"older's inequality (applied twice). That is, with $(t,\bx) \in [0,T] \times \mTx$ fixed, we have that
\[
\Big|\int_{\mRp} |\bv|\,f(t,\bx,\bv)\td^3\bv\Big|\le \Big(\int_{\mRp} f(t,\bx,\bv)\td^3\bv\Big)^{1\over2} \, \Big(\int_{\mRp}  |\bv|^2f(t,\bx,\bv)\td^3\bv \Big)^{1\over2},
\]
and therefore, for $t \in [0,T]$,
\[
\begin{aligned}
\Big\|\int_{\mRp} |\bv|\,f(t,\cdot,\bv)\td^3\bv\Big\|^{5\over4}_{L^{5\over4}(\mTx)}&\le\int_{\mTx} \Big(\int_{\mRp} f(t,\bx,\bv)\td^3\bv\Big)^{5\over8} \Big(\int_{\mRp}  |\bv|^2f(t,\bx,\bv)\td^3\bv \Big)^{5\over8}\td^3\bx\\
&\le\left[\int_{\mTx}\Big(\int_{\mRp} f(t,\bx,\bv)\td^3\bv\Big)^{{5\over8}\ccdot{8\over3}}\td^3\bx \right]^{3\over8} \, \left[\int_{\mTx}\Big(\int_{\mRp}  |\bv|^2f(t,\bx,\bv)\td^3\bv \Big)^{{5\over8}\ccdot{8\over 5}}\td^3\bx\right]^{5\over8}\\
&=  \Big\|\int_{\mRp} f(t,\cdot,\bv)\td^3\bv\Big\|_{L^{5\over3}(\mTx)}^{5\over8}\, \Big(\Epar {f}(t)\Big)^{5\over8} .
\end{aligned}
\]
Substituting \eqref{bd:p0f} into the right-hand side of the last inequality then yields \eqref{bd:pf}.

An alternative proof of \eqref{bd:pf} proceeds similarly to that of \eqref{bd:p0f}: for any $N>0$ and any $(t,\bx) \in [0,T] \times \mTx$, we have that
\[
\begin{aligned}
\left|\int_{\mRp} |\bv|\,f(t,\bx,\bv)\td^3\bv\right|&\le\int_{|\bv|\le N}|\bv|\,f(t,\bx,\bv) \td^3\bv+\int_{|\bv|> N}|\bv|\,f(t,\bx,\bv) \td^3\bv\\
&\le CN^4\|f(t,\cdot,\cdot)\|_{L^\infty(\mTx \times \mRp)}+N^{-1}\int_{\mRp} |\bv|^2f(t,\bx,\bv)\td^3\bv,
\end{aligned}
\]
which, upon choosing $N$ so that the right-hand side of the last inequality attains its minimum, and then considering the $L^{5\over4}(\mTx \times \mRp)$ norm of the expression on the left-hand side of the resulting inequality, again yields \eqref{bd:pf}.
Either approach can be adapted to prove both \eqref{bd:ps} and the more general inequality \eqref{bd:pkf}, of which \eqref{bd:p0f}--\eqref{bd:ps} are special cases
for $k=0,1,2$, respectively.
\end{proof}

\begin{remark} Similar estimates hold if we replace the $L^\infty$ norm on the right-hand side of \eqref{bd:pkf}, with a general $L^r$ norm, but we shall not use bounds of this type in our proofs and we therefore omit the details of their derivation.
\end{remark}

\section{\bf Mollified PDE system: existence of solutions via a fixed point method}\label{sec:mo}

We shall assume throughout this section that the mollification parameter $\ep$ is fixed and $0<\ep\ll 1$.
We consider a nonnegative radially symmetric function $\theta_0 \in \bC^\infty_c(\mR^3)$ such that
\[\theta_0(\bx)=0\text{\quad for any } |\bx|>{1\over2},\quad\text{ and \quad}\int_{|\bx|\le{1\over2}}\theta_0(\bx)\td^3\bx=1.
 \]
The support of the function  $\bx \mapsto \ep^{-3}\theta_0(\ep^{-1}\bx)$ is then contained in the box domain $\big\{\bx: |\bx|_\infty\le { \ep\over2} \big\}$; let $\thx$ denote the  $2\pi$-periodic extension of this function -- we shall henceforth consider $\thx(\bx)$ for $\bx \in \mTx$
only.

The mollification of a $2\pi$-periodic locally integrable function $v(\bx)$ is defined by convolution with $\thx$ and is denoted by the superscript $\langle\ep\rangle$; i.e.,
\[v\mo(\bx):=\int_{\mTx}v(\by)\,\thx(\bx-\by)\td^3 \by.\]
We will use the following property of mollification, which is a consequence of the differentiation properties of convolution and H\"older's inequality:
\be
\label{mo:prop}
\| v\mo  \|_{\bC^m(\mTx)}\le C_{\ep,r,m }\|v \|_{L^r(\mTx)}\;\;\text{ for }\;\;1\le r\le\infty,\;\;0\le m<\infty,
\ee
where $C_{\ep,r,m} = \|D^m_{\bx}\thx\|_{L^{r'}(\mTx)}$, $\frac{1}{r}+\frac{1}{r'}=1$.
%

We also introduce the Banach space of (weakly) divergence-free, square-integrable vector field pairs
\[ \incL(\mTx;\mR^6):=\Big\{(\vu,\vB)\in L^2(\mTx;\mR^6)\,\Big|\,\int_{\mTx}\vu\cdot \nabla_{\bx} \phi \td^3\bx=\int_{\mTx}\vB\cdot \nabla_{\bx} \phi \td^3\bx=0\;\;\text{ for all }\;\phi\in \bC^1(\mTx;\mR)\Big\},\]
which we equip with the usual $L^2$ norm. For notational simplicity,
when such functions appear within the $L^2$ norm sign, the superscript ``inc'' will be omitted from our notation for the norm.

\subsection{Definition of the mollified mapping}

We proceed by defining a mapping
$$\cF={\cF_{\mr\fn,\mr\un,\mr\Bn}}: \bC([0,T];\incL(\mTx;\mR^6)) \mapsto \bC([0,T];\incL(\mTx;\mR^6))$$
as follows. 
Note that $\cF$ depends on the initial data $(\mr\fn,\mr\un,\mr\Bn)$, which are considered as being fixed throughout Section \ref{sec:mo}; therefore the dependence of $\cF$ on the initial data will be, for the sake of brevity,
usually omitted from our notation. 

Given   $(\unl,\Bnl ) \in \bC([0,T];\incL(\mTx;\mR^6))$, which are $2\pi$-periodic with respect to $\bx$,
\be
\label{cS:def}
\text{let }\quad (\un,\Bn)(t,\bx)=\cF(\unl,\Bnl)={\cF_{\mr\fn,\mr\un,\mr\Bn}}(\unl,\Bnl)  \quad \text{  solve   \eqref{S:1}, \eqref{S:2}}
\ee
in the sense that
\begin{subequations}
\label{S:1}
\begin{align}
\label{f:S}
&\pt \fn+\bv\cdot \nabla_{\bx} \fn={(( \utt -\bv)\times\Btt)\cnp \fn},   \\
\label{f:S:in}
\text{with  }&\text{$\Cinf_c(\mTx \times \mRp;\mR^1)$ initial datum:}\quad  \fn\IC=\mr\fn \geq 0,
\end{align}
\end{subequations}
where $\fn$ is assumed to be $2\pi$-periodic with respect to $\bx$ for all $(t,\bv) \in [0,T] \times \mRp$, while
$\mr \fn$ is assumed to be $2\pi$-periodic with respect to $\bx$  and compactly supported in $\mTx\times\mRp$;
\begin{subequations}
\label{S:2}
\begin{alignat}{2}
\nonumber
&\pt\un+(\utt\cdot \nabla_{\bx})\un-(\Btt\cdot \nabla_{\bx}) \Bn- \Delta_{\bx}\un+\nabla_{\bx} \overline{\mathcal{P}}\\
\label{u:S}
&\qquad = \un\times\Btt\Big(\intp{\fn\,} \Big)
{+}\Big( \Btt \times\intp{\bv\,\fn\,} \Big)\mo
&&\qquad\text{(subject to }\nc\un=0); \\
\label{B:S}
&\pt \Bn+(\utt\cdot \nabla_{\bx})\Bn-(\Btt\cdot \nabla_{\bx}) \un- \Delta_{\bx}\Bn+{\nabla_{\bx} \overline{\mathcal{P}_\vB}}=0
&&\qquad \text{(subject to }\nc\Bn=0);\\
\label{uB:S:in}
&\qquad \,\text{with $2\pi$-periodic divergence-free $\Cinf(\mTx;\mR^6)$ initial data:}&&\quad(  \un, \Bn)\IC=(\mr\un,\mr\Bn),
\end{alignat}
\end{subequations}
and with  $(\un ,\Bn)$ subject to $2\pi$-periodic boundary conditions with respect to $\bx$.
This system may be solved as follows: first we solve for $\fn$ from the linear equation \eqref{S:1}; we then treat $\fn$ as given and solve for $(\un ,\Bn)$ from the linear system \eqref{S:2}. Once $(\fn,\un ,\Bn)$ is found (in a function space to be made precise in Lemma \ref{thm:S:bound} below), the auxiliary pressure variables $\overline{\mathcal{P}}$, {$\overline{\mathcal{P}_\vB}$} can be recovered by a standard procedure (for example, taking divergence of \eqref{u:S}, applying $\nc\un=0$ and inverting the Laplacian gives $\overline{\mathcal{P}}$ 
uniquely up to a constant
), and are henceforth not considered as part of the solution.

We note that no decay hypotheses need to be imposed on $\fn$ when $|\bv| \rightarrow \infty$, since (as will be shown below)  the assumption that $\mr\fn$ is   compactly supported  guarantees that $\fn(t,\cdot,\cdot)$ is also  compactly supported for all $t \in (0,T]$.

We also note that the mollification in the second product on the right-hand side of \eqref{u:S}
is intentional. It is to ensure that one can conveniently estimate the energy exchange
between \eqref{S:1} and \eqref{S:2}, and eventually eliminate from the energy equality the terms
representing energy exchanges; c.f. Remark \ref{re:en:cons} below.

\begin{lemma}\label{thm:S:bound}
Consider any $T>0$ that is independent of $\ep$. Let  $(\unl,\Bnl)\in \bC([0,T];\incL(\mTx;\mR^6))$. Then, the system \eqref{S:1}, \eqref{S:2},  subject to the initial conditions specified therein,
admits a solution $(\fn,\un,\Bn)$ that satisfies
\be
\label{smooth:cS:uBf}
\fn\in\bC([0,T];\bC^m(\mTx\times \mRp))\quad\text{and} \quad (\un,\Bn) \in \bC([0,T]; \bC^m(\mTx;\mR^6))\quad \mbox{for all $m\geq 0$}.
\ee
Moreover,
\begin{itemize}
\item $\fn \geq 0$, and there exists a scalar-valued mapping $G (\cdot,\cdot,\cdot)$ such that it is monotonically increasing with respect to its first and third arguments, and
\be
\label{fn:comp:S}
\qquad\fn(t,\bx,\bv)=0\;\;\text{ for all }\;\;|\bv|  > G \left(T,\, \mr\fn,\,\max_{[0,T]\times\mTx}\big\{|\utt|,|\Btt|\big\} \right)\;\; \mbox{and all
$(t,\bx) \in [0,T] \times \mTx$}; \ee
\item  we have, for all $t \in [0,T]$, that
\be
\label{fn:Lr:S}
\|\fn(t)\|_{L^r(\mTx\times \mRp)}=\| \mr \fn\|_{L^r(\mTx\times \mRp)} ,\quad r\in[1,\infty];
\ee
\item and the following energy equality holds for all $t \in [0,T]${\rm :}
\be
\label{energy:cS}
\begin{aligned}
\qquad{1\over2}\big(\|\un(t) \|^2_{L^2(\mTx)}+\|\Bn(t) \|^2_{L^2(\mTx)}\big)&+\Epar {\fn}(t)+\int_0^t \big(\|\nabla_{\bx}\un(s) \|^2_{L^2(\mTx)}+\|\nabla_{\bx}\Bn(s) \|^2_{L^2(\mTx)}\big)\td s \\&= {1\over2}\big(\|\mr \un \|^2_{L^2(\mTx)}+\| \mr\Bn \|^2_{L^2(\mTx)}\big)+\Epar { \mr \fn}  +\int_0^t\Rn(s)\td s,\\
   \text{where \;\;}\Rn(s):= \int_{\mTx}& ( \un(s,\bx)  -\unl(s,\bx) )\mo\cdot\Big( \Btt(s,\bx) \times\intp{\bv\,\fn(s,\bx,\bv)}\Big)\td^3\bx,
\end{aligned}
\ee
with the notation $\Epar{\cdot}$ defined in
\eqref{def:Ef}.
\end{itemize}
\end{lemma}

\begin{proof} First, by the properties of the mollifier \eqref{mo:prop}, we have that
\[
\text{\CCm{(\utt,\Btt)}.}
\]
To solve the equation \eqref{f:S}, we first solve the following family of ordinary differential equations for the
associated characteristic curves:
\be
\label{char:eq}
\left\{\begin{aligned}\pt\xx(t;{\bx}_0,{\bv}_0) &=\pp(t;{\bx}_0,{\bv}_0),&&\xx(0;{\bx}_0,{\bv}_0)={\bx}_0, \\
\pt\pp (t;{\bx}_0,{\bv}_0)&=-( \utt(t,\xx(t;{\bx}_0,{\bv}_0))-\pp(t;{\bx}_0,{\bv}_0)\times\Btt(t,\xx(t;{\bx}_0,{\bv}_0)),&&\pp(0;{\bx}_0,{\bv}_0)={\bv}_0,
\end{aligned}\right.
\ee
with  $(\xx,\pp)\in \mT^3 \times\mR^3$. As \CCm{(\utt , \Btt)}, the right-hand side of \eqref{char:eq} and their $D_{\xx, \pp},\, D_{\xx, \pp}^2,\ldots$ derivatives are continuous and grow at most linearly in $|\pp|$; in fact, they are bounded by   $a|\pp|+b$, where $a,b$ only depend on  suitable $\bC([0,T]; \bC^m(\mTx;\mR^6))$ norms  of $(\utt , \Btt)$.
Then, by applying classical results from the theory of ordinary differential
equations to the system \eqref{char:eq} and to its $D_{{\bx}_0,{\bv}_0},\,D^2_{{\bx}_0,{\bv}_0},\ldots$ derivatives, we deduce that, for any initial data $({\bx}_0,{\bv}_0) \in \mTx \times \mRp$, \eqref{char:eq} admits a unique   solution, all of whose $D_{{\bx}_0,{\bv}_0}$, $D^2_{{\bx}_0,{\bv}_0},\ldots$ derivatives are continuous functions defined in $[0,T]\times\mTx\times\mRp$ (see, for example, \cite[Corollary 4.1 on p.101]{hartman}). Moreover, this reasoning is time-reversible, so that for any $t\in[0,T]$, $(\xx(t;\cdot,\cdot),\pp(t;\cdot,\cdot))$ has a unique inverse, which we denote by  $(\xx^{-t}(\cdot,\cdot),\pp^{-t}(\cdot,\cdot))$,  all of whose $D_{\bx,\bv},\,D^2_{\bx,\bv},\ldots$ derivatives are continuous in $[0,T]\times\mTx\times\mRp$. All in all,
\[
\fn(t,\bx,\bv):= \mr \fn\left(\xx^{-t}(\bx,\bv),\pp^{-t}(\bx,\bv)\right) \geq 0
\]
solves the equation \eqref{S:1}, and $D_{x,p}^m\fn\in  \bC ([0,T];\bC(\mTx \times \mRp))$ for all $m\ge0$.

Furthermore, since $\mr\fn$ has been assumed to have compact support in $\mTx \times \mRp$, there exists a positive real number $C_1=C_1(\mr\fn)$ such that
\be
\label{fn:0}
\fn(0,\bx,\bv)=0\qquad \mbox{for all $|\bv|>C_1$ and all $\bx \in \mTx$}.
\ee
By the second equation of \eqref{char:eq},  we have, with $C_2:=\max_{[0,T]\times\mTx}\big\{|\utt|,|\Btt|\big\}$, 
\[
{1\over2} \frac{{\td}}{\td t}|\pp(t;{\bx}_0,{\bv}_0) |^2 =  -\pp \cdot\Big((\utt  \times\Btt)(t,\xx(t;{\bx}_0,{\bv}_0 ))\Big)\le  {1\over2} |\pp |^2 +{1\over2} C_2^4,
\]
which then implies, for any ${\bv}_0\in \mRp$ such that $|{\bv}_0|\le C_1$, that
\[
|\pp(t;{\bx}_0,{\bv}_0) |^2 \le {C_1^2{\rm e}^{T} }+({\rm e}^{T}-1) C_2^4=:G^2 \qquad \forall\,t \in [0,T].
\]
Together with \eqref{fn:0}, this immediately implies that $\fn(t,\bx,\bv)=0$ for all $\bv \in \mRp$ such that $|\bv|>G $
and all $(t,\bx) \in [0,T] \times \mTx$, as has been stated in \eqref{fn:comp:S}.  The  monotonicity  of $G =G \big(T,\, \mr\fn,\,\max_{[0,T]\times\mTx}\big\{|\utt|,|\Btt|\big\} \big)$ as stated above \eqref{fn:comp:S} also follows.

Next we shall prove \eqref{fn:Lr:S}. The construction of $\fn$ implies that the function $t \in [0,T] \mapsto \|\fn(t,\cdot,\cdot)\|_{L^\infty(\mTx \times \mRp)}$ is constant. For $r\in[1,\infty)$, the fact that  $t \in [0,T] \mapsto \|\fn(t,\cdot,\cdot)\|_{L^r(\mTx \times \mRp)}$ is constant follows from integrating $|\fn|^{r-1}$ times \eqref{f:S} over $\mTx\times\mRp$ and using that $\nabla_{\bx}\ccdot \bv=\nabla_{\bv}\ccdot(( \utt -\bv )\times\Btt)=0$ together with the divergence theorem, the compact support of $\fn$, and the $2\pi$-periodicity of $\fn$ with respect to $\bx$.

Next,  we substitute  $(\fn,\utt,\Btt)$ into the linear(ized) MHD system \eqref{S:2}. Thanks to the smoothness and compactness of the support of $\fn(t,\cdot,\cdot)$,  the hypotheses of Lemma \ref{thm:linearMHD} are satisfied. Consequently, we have that \CCm{(\un,\Bn)}.

To prove the energy equality, we apply \eqref{energy:fMHD} of Lemma \ref{thm:linearMHD} to deduce that, for any $t\in[0, T ]$,
\be
\label{energy:S:MHD}
\begin{aligned}
&{1\over2}(\|\un(t) \|^2_{L^2(\mTx)}+\|\Bn(t)\|^2_{L^2(\mTx)})+\int_0^t \big(\|\nabla_{\bx}\un(s) \|^2_{L^2(\mTx)}+\|\nabla_{\bx}\Bn(s) \|^2_{L^2(\mTx)}\big)\td s\\
&\qquad={1\over2} (\| \un(0) \|^2_{L^2(\mTx)}+\| \Bn(0) \|^2_{L^2(\mTx)}) +\int_0^t \int_{\mTx}{ \un\mo  \cdot\Big( \Btt \times\intp{\bv\,\fn}\Big)}\td^3\bx\td s.
\end{aligned}
\ee
Note that in the last term we transferred the mollifier $\mo$ onto $\un$.

Also, by multiplying  \eqref{f:S} with $ |\bv|^2$, integrating over $\mTx\times\mRp$, and performing integration by parts (which is justified, since $\fn(t,\bx,\cdot)$ is compactly supported in $\mRp$ and $\fn(t,\cdot,\bv)$ is $2\pi$-periodic in $\bx$), we obtain
\[
\begin{aligned}
{\ddt}\int_{\mTx\times \mRp}  {1\over2}|\bv|^2\, \fn\td^3\bv\td^3\bx&=\int_{\mTx\times \mRp}  {1\over2} |\bv|^2 \,\nabla_{\bv}\ccdot((\utt-\bv)\times\Btt \fn)\td^3\bv\td^3\bx\\
&=-\int_{\mTx\times \mRp}  \bv \ccdot((\utt-\bv)\times\Btt \fn)\td^3\bv\td^3\bx
\\&=- \int_{\mTx}    \left({ \int_{\mRp} \bv\,\fn\td^3\bv}\right)\cdot (\utt\times\Btt)\td^3\bx \\&=-\int_{\mTx}  { \utt}   \cdot\Big( \Btt \times\intp{\bv\,\fn}\Big)\td^3\bx.
\end{aligned}
\]
We then integrate this equality from $0$ to $t \in (0,T]$ and add \eqref{energy:S:MHD} to it to complete the proof of \eqref{energy:cS}.
\end{proof}

\subsection{Verification of the hypotheses of Schauder's fixed point theorem}

Having shown that the mapping $\cF$ is correctly defined, we shall next apply the following version of Schauder's fixed point theorem.
\begin{theorem}[Schauder's fixed point theorem]\label{Schauder}
Suppose that $\sK$ is a convex subset of a topological vector space and $\cF$ is a continuous mapping of $\sK$
into itself such that the image $\cF(\sK)$ is contained in a compact subset of $\sK$; then, $\cF$ has a fixed point.
\end{theorem}
Here and below, compactness is in the strong sense, unless stated otherwise.
\begin{remark}\label{re:en:cons}
When $\cF$ admits a fixed point, i.e., $(\unl,\Bnl)=(\un,\Bn)=\cF(\unl,\Bnl)$, the $R$ term in  \eqref{energy:cS}  vanishes, and one recovers the usual energy law \eqref{ddt:Einc} for the \emph{mollified} system. Our next objective is therefore to show, by applying Schauder's fixed point theorem, that the mollified system has a solution. Once we have done so, we shall pass to the limit $\ep \rightarrow 0$ with the mollification parameter. For the moment however $\ep>0$ is held fixed.
\end{remark}

Recall the compact support result \eqref{fn:comp:S} with the monotonicity property of $G(\cdot,\cdot,\cdot)$ specified therein. Together with  the property of mollification $\|(\unlm,\Bnlm)\|_{\bC([0,T];\bC^m(\mTx))}\le C_{\ep,m}\|(\unl,\Bnl)\|_{\bC([0,T]; L^2(\mTx))}$, this yields
\be
\label{comp:fn:cS}
\fn (t,\bx,\bv)=0\;\;\text{ for \; } |\bv|\ge G\left ( T,\mr\fn,C_{\ep,1}\|(\unl,\Bnl)\|_{\bC([0,T]; L^2(\mTx))}\right),\quad\text{for any admissible $\mr\fn$}\;\text{and}\; t\in[0,T].
\ee

The next three lemmas are concerned with verifying the hypotheses of Schauder's fixed point theorem for the mapping $\cF$ defined in \eqref{cS:def}.
We begin, in the next lemma, by proving the continuity of $\cF$.
\begin{lemma}
\label{lm:cS:cont}
For any $T>0$ that is independent of $\ep$, the mapping $\cF$ defined in \eqref{cS:def} subject to fixed initial data $(\mr\fn,\mr\un,\mr\Bn)$ is continuous from the Banach space $\bC([0,T];\incL(\mTx;\mR^6))$
into itself.
\end{lemma}
\begin{proof}
For $i=1,2$, consider $(\unl_i,\Bnl_i)\in\bC([0,T];\incL(\mTx;\mR^6))$ and the associated  $(\un_i,\Bn_i)$ and $\fn_i$. Clearly, Lemma \ref{thm:S:bound}  guarantees that, for $i=1,2$, $(\un_i,\Bn_i)\in\bC([0,T];\incL(\mTx;\mR^6))$ and $\fn_i,\un_i,\Bn_i$ are smooth.

Suppose further that
\be\label{de:u12:B12}
\delta:=\left\|(\unl_1-\unl_2,\Bnl_1-\Bnl_2)\right\|_{\bC([0,T];L^2(\mTx))}\ll1.
\ee
Let $(\unl_1,\Bnl_1)$ be fixed so that $ (\fn_1,\un_1,\Bn_1)$ are all fixed.
Then, our goal is to show that
$(\un_2,\Bn_2)\to(\un_1,\Bn_1)$  strongly in  $\bC([0,T];L^2(\mTx;\mR^6))$ as $\delta\to0$.

To this end, let \[\fn_{12}:=\fn_1-\fn_2 \quad\text{ and likewise for }  \un_{12},\,\Bn_{12} .\]
First, by \eqref{de:u12:B12} and the properties of the mollification \eqref{mo:prop}, we have that
\be
\label{cont:nlm}\lim_{\delta\to0}(\unlm_2,\Bnlm_2)=(\unlm_1,\Bnlm_1)\text{ \;\; strongly in \;}\bC([0,T]\times\mTx;\mR^6).
\ee
Next, by \eqref{S:1}, the governing equation for $\fn_{12}$ is
\[
\pt \fn_{12}+\bv\cdot \nabla_{\bx} \fn_{12}={(( \unlm_2 -\bv)\times\Bnlm_2)\cnp \fn_{12}}+\left(( \unlm_i -\bv)\times\Bnlm_i\right)\Big|_{i=2}^{i=1}\cnp \fn_1,
\]
with initial datum $\fn_{12}\IC=0$. By the method of characteristics, similarly to \eqref{char:eq} and the argument thereafter, one can show that
\[
\|\fn_{12}\|_{\bC([0,T]\times\mTx\times\mRp)}\le T \left\|\left(( \unlm_i -\bv)\times\Bnlm_i\right)\Big|_{i=2}^{i=1}\cnp \fn_1\right\|_{\bC([0,T]\times\mTx\times\mRp)}.
\]
Since  $ \fn_1 $ is fixed and smooth (c.f. \eqref{smooth:cS:uBf}), thanks to the compactness of its support, as specified in \eqref{comp:fn:cS},  and the above estimate together with $\fn_{12}\IC=0$, we have that $\fn_{12}$ also has compact support in $[0,T]\times\mTx\times\mRp$ that is independent of $\delta$. Therefore,  we combine the last estimate and \eqref{cont:nlm} to obtain
\be
\label{cont:fn:moment}
\lim_{\delta\to0}  \intp{\fn_{12}\,}=0\quad\mbox{and}\quad\lim_{\delta\to0}\intp{\bv\,\fn_{12}\,}=\mathbf{0},
\ee
strongly in $\bC([0,T]\times\mTx)$ and $\bC([0,T]\times\mTx;\mathbb{R}^3)$, respectively.

We move on to \eqref{S:2} and write the governing equations for $(\un_{12},\Bn_{12})$ as
\begin{subequations}
\label{S:2:cont}
\begin{alignat}{2}
\nonumber
&\qquad\pt\un_{12}+\unlm_2\cdot \nabla_{\bx}\un_{12}-\Bnlm_2\cdot \nabla_{\bx} \Bn_{12}- \Delta_{\bx}\un_{12}-\un_{12}\times\Bnlm_2\intp{\fn_2}
 +\nabla_{\bx} \overline{\mathcal{P}_{12}}\\
\label{cont:u}
&\qquad\qquad=-\Big( \Bnlm_i \times\intp{\bv\,\fn_i}\Big|_{i=2}^{i=1} \Big)\mo  -\unlm_{12}\cdot \nabla_{\bx}\un_{1}+\Bnlm_{12}\cdot \nabla_{\bx}\Bn_{1}+\un_{1}\times\Bnlm_i\intp{\fn_i}\Big|_{i=2}^{i=1}\\
&\nonumber \hspace{4in}\text{(subject to }\nc\un_{12}=0),\\
\label{cont:B}
&\qquad\pt \Bn_{12}+\unlm_2\cdot \nabla_{\bx}\Bn_{12}-\Bnlm_2\cdot \nabla_{\bx} \un_{12}- \Delta_{\bx}\Bn_{12}+{\nabla_{\bx} \overline{(\mathcal{P}_\vB)_{12}}}=-\unlm_{12}\cdot \nabla_{\bx}\Bn_{1}+\Bnlm_{12}\cdot \nabla_{\bx} \un_1\\
&\nonumber\hspace{4in} \text{(subject to }\nc\Bn_{12}=0),
\end{alignat}
\[
\text{with initial data}\quad(  \un_{12}, \Bn_{12})\IC=(\mathbf{0},\mathbf{0}).
\]
\end{subequations}
Since $(\un_{1},\Bn_{1})$ is fixed and smooth, we use \eqref{cont:nlm} and \eqref{cont:fn:moment} to deduce that
\be\label{RHS:conv:0}
  \text{the right-hand sides of \eqref{cont:u}, \eqref{cont:B} converge to $\mathbf{0}$ strongly in }  \bC([0,T]\times\mTx;\mathbb{R}^3)\,\text{ as }\delta\to0.
\ee
Finally, we invoke \eqref{energy:fMHD} of Lemma \ref{thm:linearMHD} and use the Cauchy--Schwarz inequality to bound the right-hand side of   \eqref{energy:fMHD}   by $ T\,\!\displaystyle\max_{[0,T]}\big(\|\vu\|_{L^2(\mTx)}\|\bh\|_{L^2(\mTx)}+\|\vB\|_{L^2(\mTx)}\|\bh_1\|_{L^2(\mTx)}\big)$. We then combine this with \eqref{S:2:cont} and \eqref{RHS:conv:0} to finally deduce that 
\[
\lim_{\delta\to0}(\un_2,\Bn_2)=(\un_1,\Bn_1)\text{ \;\; strongly in \;}\bC([0,T];L^2(\mTx;\mR^6)).
\]
That completes the proof of the lemma.
\end{proof}

Next, we will verify the endomorphism hypothesis in Schauder's fixed point theorem, for which the energy equality \eqref{energy:cS} plays a key role.
For $a>0$, we introduce the following convex set:
\[
\sK[T,a]:=\Big\{(\vu,\vB)\in \bC([0,T]; \incL(\mTx;\mR^6)) \,\Big|\,\|(\vu,\vB)\|^2_{\bC([0,T]; L^2(\mTx))}\le a \Big\} .
\]
We also recall the definition of $\Epar{\cdot}$ stated in \eqref{def:Ef}, and the definition 
of $\Ettinc[ f, \vu, \vB]$ at the start of Section \ref{sec:cons}. 
%
%
For brevity,  let\[\Eno:= \Ettinc[\mr\fn,\mr\un,\mr\Bn].\]
\begin{lemma}
\label{endo:lemma}
For a fixed  $\ep>0$, there exists a constant $C_\ep>0$ that only depends on $\ep$ and such that, with
\be\label{Tflat}
T^\flat=T^{\flat}\big( \ep,|\mr\fn|_{L^\infty(\mTx\times \mRp)},\Eno\big): =  C_\ep\,|\mr\fn|_{L^\infty(\mTx)}^{-{1\over5}} \,(2\Eno )^{-{4\over5}},
\ee
the mapping $\cF$, defined in \eqref{cS:def} subject to fixed initial data $(\mr\fn,\mr\un,\mr\Bn)$, maps the convex set $\sK [T^{\flat}, 4\Eno  ] \subset {\bC([0,T^{\flat}]; \incL(\mTx;\mR^6))}$ into itself. Furthermore, for any $t \in [0, T^{\flat}]$,
\be
\label{energy:Tflat}
{1\over2}\big(\|\un(t) \|^2_{L^2(\mTx)}+\|\Bn(t) \|^2_{L^2(\mTx)}\big)+\Epar {\fn}(t)+\int_0^t \big(\|\nabla_{\bx}\un \|^2_{L^2(\mTx)}+\|\nabla_{\bx}\Bn \|^2_{L^2(\mTx)}\big)\td s \le 2\,\Eno.
\ee
\end{lemma}
\begin{proof}
In order to avoid the trivial case of zero initial data, we only consider $T^\flat<\infty$. We begin by choosing any $(\unl,\Bnl)\in \sK [T^{\flat},  4 \Eno ]$, i.e, a pair that satisfies the  bound
\be
\label{uBnl:A0}
\|(\unl,\Bnl)\|_{\bC([0,T^\flat];\, L^2(\mTx))}^2\le  4\Eno.
\ee
Next, in the   definition of $\Rn$ featuring in the energy equality \eqref{energy:cS}, we transfer the first mollifier $\mo$ onto the second factor and apply the Cauchy--Schwarz inequality to obtain
\[
|\Rn|\le \|\un-\unl  \|_{L^2(\mTx)}\Big\|\big[  \Bnl\mo \times\intp{\bv\,\fn}\big]\mo\Big \|_{L^2(\mTx)}.
\]
Then, to estimate the second factor, we apply twice the property of mollification stated in \eqref{mo:prop} together with H\"older's inequality, to deduce the existence of a constant $C_\ep$, which depends on $\ep$, such that
\be
\label{interm:est1}
\begin{aligned}
&\;\;\quad\Big\|\big[  \Bnl\mo   \times\intp{\bv\,\fn}\big]\mo\Big \|_{L^2(\mTx)}\le  C_\ep \|\Bnl\|_{L^2(\mTx)}\Big\|\intp{\bv\,\fn}\Big\|_{L^{5\over4}(\mTx)} \le C_\ep\|\Bnl\|_{L^2(\mTx)}\,\Epar {\fn}^{4\over5}\,|\mr\fn|_{L^\infty(\mTx)}^{1\over5},
\end{aligned}
\ee
where  the last inequality follows from  Proposition \ref{prop:Lr} and the invariance property \eqref{fn:Lr:S}.
Therefore,
\[
|\Rn(s)| \le C_\ep\big(\|\un(s)\|_{L^2(\mTx)}+\|\unl(s)\|_{L^2(\mTx)}\big)\,\|\Bnl(s)\|_{L^2(\mTx)}\,[\Epar {\fn}(s)]^{4\over5}\,|\mr\fn(s)|_{L^\infty(\mTx)}^{1\over5},\quad s \in [0,T]. \]
Substituting this bound on $R$ into \eqref{energy:cS}, noting \eqref{uBnl:A0},
we arrive at
\[
\begin{aligned}
{1\over 2}\big(\|\un(t) \|^2_{L^2(\mTx)}+\|\Bn(t) \|^2_{L^2(\mTx)}\big)+\Epar {\fn}(t)+\int_0^t \big(\|\nabla_{\bx}\un(s) \|^2_{L^2(\mTx)}+\|\nabla_{\bx}\Bn(s)\|^2_{L^2(\mTx)}\big)\td s \\
\qquad\le {\Eno } +\int_0^t C_\ep\big(\|\un(s)\|_{L^2(\mTx)}+(4\Eno)^{1\over 2} \big)\,(4\Eno )^{1\over 2}\,[\Epar { \fn}(s)]^{4\over5}\,|\mr\fn(s)|_{L^\infty(\mTx)}^{1\over5}\td s, \quad
t \in [0,T].
\end{aligned}
\]
We see that the expression on the left-hand side of inequality \eqref{energy:Tflat} is a continuous function of $t $, whose value at $t=0$ is \textit{strictly less than} $2\Eno$ on the right-hand side of   \eqref{energy:Tflat}. For nontriviality, we only consider the case when the left-hand side of   \eqref{energy:Tflat} equals $2\Eno$ at least once, at a certain positive time, so that we can define    $T^\sharp>0$ as the \textit{earliest} time at which this happens. We then set $t=T^\sharp$ in  the above estimate, which makes  the left-hand side equal $2\Eno$, i.e.,
\[
2\Eno \le {\Eno } +\int_0^{T^\sharp}C_\ep\big(\|\un(s)\|_{L^2(\mTx)}+(4\Eno)^{1\over 2} \big)\,(4\Eno )^{1\over 2}\,[\Epar { \fn}(s)]^{4\over5}\,|\mr\fn|_{L^\infty(\mTx)}^{1\over5}\td s.
\]
The minimality of $T^\sharp$ also means that, for all $t \in [0,T^\sharp]$, the estimate ${1\over2}\|\un(t) \|^2_{L^2(\mTx)} + \Epar {\fn}(t) \leq 2\Eno $ holds.
Thus, continuing from the last inequality, we obtain
\[
\begin{aligned}
{2\Eno }
&\le {\Eno }+ {T^\sharp} C_\ep\big((4\Eno)^{1\over 2} +(4\Eno)^{1\over 2} \big)\,(4\Eno)^{1\over 2}\,(2\Eno )^{4\over5}\,|\mr\fn|_{L^\infty(\mTx)}^{1\over5},
\end{aligned}
\]
which implies that $ T^\sharp\ge T^\flat$, with $T^\flat$ as defined in \eqref{Tflat} (upon redefining $C_\ep$).
Since, $T^\sharp>0$ is the earliest time at which equality is attained in \eqref{energy:Tflat}, we have thus proved that the estimate \eqref{energy:Tflat} holds for at least all $t \in [0,T^{\flat}]$. The proof is complete.
\end{proof}

Finally, to verify the compactness condition (in the strong topology) in Schauder's fixed point theorem, we shall use the Aubin--Lions--Simon lemma (see, for example,
\cite[Corollary 4 on p.85]{simon}).

\begin{theorem}[Aubin--Lions--Simon lemma]\label{embedding:lemma}
Let $\bX_1$, $\bX_0$ and $\bX_{-1}$ be three Banach spaces with
\emph{compact embedding} $\bX_1\hookrightarrow\hookrightarrow \bX _0$ and continuous embedding $\bX _0\hookrightarrow \bX_{-1}$.  For $1\le r, s \le \infty$ and positive constants $C_1,C_2$, consider the set
\[
{\mathsf S} := \Big\{u\;\big|\; \|u\|_{L^r ((0, T); \bX_1)}\le C_1,\; \|\pt{u}\|_{ L^s ((0, T); \bX_{-1})}\le C_2 \Big\}.
\]
Then, the following statements hold:
\begin{itemize}
\item[(i)] If $r < \infty$, then   ${\mathsf S}$ is compact in  $L^r((0, T); \bX_0)$;
\item[(ii)] If $r  = \infty$ and $s >  1$, then  ${\mathsf S}$ is compact in $\bC([0, T]; \bX_0)$.
\end{itemize}
\end{theorem}
\begin{lemma}
\label{comp:lemma}
Suppose that $\ep>0$. With the same hypotheses and notations as in Lemma \ref{endo:lemma}, the image   of the convex set $\sK:=\sK\Big[T^{\flat}, 4\Eno \Big] \subset {\bC([0,T^{\flat}]; \incL(\mTx;\mR^6))}$ under $\cF$ is contained in a  compact subset of $\sK$.
\end{lemma}
\begin{proof}
By applying \eqref{energy:pt} of Lemma \ref{thm:linearMHD} to the system \eqref{S:1}, \eqref{S:2}, we deduce that, for all $t \in (0,T^\flat]$,
\be
\label{energy:pt:cS}
\begin{aligned}
&\|(\nabla_{\bx}\un,\nabla_{\bx}\Bn) \|^2_{L^2(\mTx)}\Big|_0^{t} + \int_0^{t}  \|(\pt\un,\pt\Bn) \|^2_{L^2(\mTx)} \td s\\
&\qquad \le2\max_{[0,t]\times\mTx}\{|\unlm|^2,|\Bnlm|^2,|\bg|^2,1\}\int_0^{t}\|(\nabla_{\bx}\un,\nabla_{\bx}\Bn,\un,\bh)\|_{L^2(\mTx)}^2\td s,
\end{aligned}
\ee
where
\[
\bg:=\Bnlm\intp{\fn} , \quad \bh:=\Big( \Bnlm \times\intp{\bv\,\fn} \Big) \mo .
\]
One can then derive an upper bound on the right-hand side of \eqref{energy:pt:cS} that   only depends on $\ep$ and the initial data as follows. First, combining property \eqref{mo:prop}  and the estimates \eqref{energy:Tflat}--\eqref{interm:est1} in  Lemma \ref{endo:lemma} and its proof, we establish bounds on
\[\max_{[0,t]\times\mTx}\{|\unlm|^2,|\Bnlm|^2 \}\quad\mbox{and}\quad \dint_0^{t}\|(\nabla_{\bx}\un,\nabla_{\bx}\Bn,\un,\bh)\|_{L^2(\mTx)}^2 \td s,
\]
where $0< t\le T^\flat$. It remains to bound the maximum of $|\bg|$, which requires bounding the $L^\infty$ norm of
\[\intp{\fn}\]
over $[0,T^\flat]\times\mTx$. This \textit{cannot} be done using Proposition \ref{prop:Lr}; instead, one can obtain the desired bound by recalling the compact support \eqref{comp:fn:cS} and the $L^\infty_{\bx,\bv}$ invariance property \eqref{fn:Lr:S}.

All in all, we have obtained bounds on $\|(\un,\Bn)\|_{L^\infty((0,T^\flat);H^1(\mTx))}$ and $\|\pt(\un,\Bn)\|_{L^2((0,T^\flat);L^2(\mTx))}$ that only depend on $\ep$ and the initial data. Therefore, by item (ii) of Theorem \ref{embedding:lemma} and recalling the compact Sobolev embedding $H^1(\mTx)\hookrightarrow\hookrightarrow L^2(\mTx)$, we complete the proof of the lemma.
\end{proof}

\subsection{Classical solution of the mollified system}

We have thus shown that $$\cF={\cF_{\mr\fn,\mr\un,\mr\Bn}}\,:\,\bC([0,T^{\flat}]; \incL(\mTx;\mR^6)) \rightarrow \bC([0,T^{\flat}]; \incL(\mTx;\mR^6)),$$ defined in \eqref{cS:def}, satisfies the three hypotheses of Theorem \ref{Schauder}, which are:
\begin{itemize}
\item continuity (by Lemma \ref{lm:cS:cont});
\item endomorphism (by Lemma \ref{endo:lemma}); and
\item compactness (by Lemma \ref{comp:lemma}).
\end{itemize}
We therefore deduce from Theorem \ref{Schauder} that $\cF$ has a fixed point $(\un,\Bn)$ in the space $\bC([0,T^{\flat}]; \incL(\mTx;\mR^6))$, where $T^\flat$ is no less than  (recalling \eqref{Tflat})
\[
T^\flat=T^{\flat}\big( \ep,|\mr\fn|_{L^\infty(\mTx\times \mRp)},\Eno\big)= C_\ep \,|\mr\fn|_{L^\infty(\mTx)}^{-{1\over5}} \,(2\Eno )^{-{4\over5}}>0,
\]
and also $(\un,\Bn)$ and the associated $\fn$ satisfy the smoothness properties \eqref{smooth:cS:uBf}. By the left-continuity of the mapping
$t \in [0,T^\flat] \mapsto (\un(t),\Bn(t)) \in L^2(\mTx;\mR^6)$, we can repeat the same argument inductively for the time intervals $[nT^\flat, (n+1)T^\flat]$  of equal length $T^\flat$, for $n=0,1,2,\ldots, [T/T^\flat]+1$, in order to reach the endpoint $T$ of the time interval $[0,T]$. The success of this   inductive process is guaranteed by  the following facts, which are independent of $n$:
\begin{itemize}
\item $\fn$ remains compactly supported thanks to \eqref{comp:fn:cS};
\item the first argument $\ep$ of $T^\flat(\cdot,\cdot,\cdot)$ is fixed and the second argument $|\mr\fn|_{L^\infty(\mTx\times \mRp)}$ of $T^\flat(\cdot,\cdot,\cdot)$ is also constant thanks to \eqref{fn:Lr:S}; and
\item the third argument $\Eno$ of $T^\flat(\cdot,\cdot,\cdot)$ is \textit{nonincreasing} (as the induction step $n$ increases) thanks to the energy equality \eqref{energy:cS} with fixed point $\unl=\un$ so that $R\equiv0$; therefore, the value of $T^\flat$ is \textit{nondecreasing} and we can thus fix it as its initial value at $n=0$ without affecting the final conclusion.
\end{itemize}

Having shown the existence of a fixed point in $\bC([0,T];\incL(\mTx;\mR^6))$ for the mapping $\cF={\cF_{\mr\fn,\mr\un,\mr\Bn}}$, we can set $(\un,\Bn)=(\unl,\Bnl )=(\ue,\Be)$ and $\fn=\fe$, which also makes   $(\utt,\Btt)=(\uee,\Bee)$,   in \eqref{S:1}, \eqref{S:2} and associated results (especially in Lemma \ref{thm:S:bound}) to deduce the following main result of this section.

\begin{theorem}\label{thm:mo}
Consider the following mollified hybrid Vlasov-MHD system, with fixed $\ep>0$:
\begin{subequations}
\label{e:1}
\begin{align}
\label{f:e}
&\pt \fe+\bv\cdot \nabla_{\bx} \fe={\big(( \uee -\bv)\times \Bee \big)\cnp \fe},   \\
\label{f:e:in}\text{with}&\text{ initial datum}\quad  \fe\IC=\mr\fe\in{\bC^\infty_c}(\mTx \times \mRp),
\end{align}
where $\mr\fe$ is  compactly supported in $  \mTx\times\mRp$; and
\end{subequations}
\begin{subequations}
\label{e:2}
\begin{align}
\nonumber\pt\ue+ (\ue\mo \cdot \nabla_{\bx})\ue- (\Be\mo \cdot \nabla_{\bx}) \Be&- \Delta_{\bx}\ue+\nabla_{\bx}  {\mathcal{P}}\\
\label{u:e}= \ue\times \Be\mo \Big(\intp{\fe} \Big) {+}&\Big(  \Be\mo  \times\intp{\bv\fe} \Big)\mo  \qquad  \;\text{(subject to }\nc\ue=0), \\
\label{B:e}\pt \Be+ (\ue\mo \cdot \nabla_{\bx})\Be- (\Be\mo \cdot \nabla_{\bx}) \ue&- \Delta_{\bx}\Be+{\nabla_{\bx}\mathcal{P}_\vB}=0\qquad  \;\, \text{(subject to }\nc\Be=0),\\
\label{uB:e:in} \text{with {$2\pi$-periodic divergence-free}  initial data}&\quad(  \ue, \Be)\IC=(\mr{\mathbf{U}}_\varepsilon,\mr{\mathbf{B}}_\varepsilon)\in\bC^\infty(\mTx).
\end{align}
\end{subequations}
Then, for any $T>0$, the above system admits a classical solution
\[
\fe\in\bC([0,T];\bC^m(\mTx \times \mRp))\quad\text{      and \quad \CCm{(\ue,\Be)}.}
\]
Moreover, for all $t \in (0,T]$, we have the invariance
\be
\label{fn:Lr:mollify}
\|\fe(t)\|_{L^r(\mTx \times \mRp)}=\| \mr \fe\|_{L^r(\mTx \times \mRp)} ,\quad r\in[1,\infty];
\ee
and the following energy equality holds:
\be
\label{energy:mollify}
\Ettinc[\fe,\ue,\Be](t)  +\|(\nabla_{\bx}\ue,\nabla_{\bx}\Be)\|^2_{L^2([0,t]; L^{2}(\mTx))}= \Ettinc[\mr\fe,\mr{\mathbf{U}}_\varepsilon,\mr{\mathbf{B}}_\varepsilon],
\ee
where
\[\Ettinc[f,\vu,\vB](t): ={1\over2}\big(\|\vu(t) \|^2_{L^2(\mTx)}+\|  \vB(t) \|^2_{L^2(\mTx)}\big) +\Epar{f}(t).\]
\end{theorem}
In particular, the remainder term $R$ in \eqref{energy:cS} vanishes from the   energy equality \eqref{energy:mollify} thanks to the fixed point property $\un=\unl=\ue$. 
Also, due to the same reasoning as below \eqref{S:2}, the auxiliary variables $\mathcal{P},\mathcal{P}_\vB$ are  not considered to be part of the solution.
\section{\bf Proof of the main result: global existence of  weak solutions}\label{sec:compact}

In this section we prove the main result of this article,  Theorem \ref{thm:incomp},  by showing that, as $\ep\to0$, a subsequence of $\{(f_\ep,\vu_\ep,\vB_\ep)\}_{\ep>0}$ that solves the mollified system formulated in the previous section converges to a weak solution $(f ,\vu ,\vB)$ that solves the original incompressible hybrid Vlasov-MHD model \eqref{eq}, and that this weak solution exists globally in time, i.e., for all nonnegative times.

Throughout this section,  the initial data of   the mollified system \eqref{e:1}, \eqref{e:2} will be constructed from the original initial data, \eqref{initial}, as follows:
\be
\label{e:in:def}
\mr{\mathbf{U}}_\varepsilon=\mr\vu\mo,\quad\mr{\mathbf{B}}_\varepsilon=\mr\vB\mo,\quad \mr\fe=\big(\mr f\ccdot\mychi{|\bv|\le1/\ep}(\bv)\big)*(\thx\,\thp).
\ee
Here, $\mychi{|\bv|\le1/\ep}$ is the cut-off function taking the value 1 in the ball of radius $1/\ep$ in $\mRp$ and 0 otherwise, and the mollifier $\thp=\thp(\bv):=\ep^{-3}\theta_0(\ep^{-1}\bv)$ for $\bv\in\mRp$  where $\theta_0$ (and also $\thx$) has been defined  at the start of Section \ref{sec:mo}.

We will   work with the weak formulation of the mollified system \eqref{e:1}, \eqref{e:2}, which is constructed as follows. We fix any $T>0$. The test functions used in this section are: scalar-valued,  compactly supported functions $g \in\bC^1_c([-1,T+1]\times\mTx\times\mRp)$, and  $\mR^3$-valued functions $\vv \in\bC^1([-1,T+1]\times\mTx;\mathbb{R}^3)$ satisfying $\nabla_{\bx}\ccdot \vv=0$.

For any $t\in (0,T]$, we multiply \eqref{f:e} by $g$ and integrate over $[0,t]\times\mTx\times\mRp$, and we take the dot product of \eqref{u:e} and \eqref{B:e} with $\vv$, and then integrate both over  $[0,t]\times\mTx$.
By performing integrations by parts, noting that
\[
\nabla_{\bv}\ccdot \big(( \ue\mo -\bv)\times \Be\mo \big)=\ncx \ue\mo =\ncx \Be\mo =0,
\]
and observing that, because of the periodic boundary conditions with respect to $\bx$ and thanks to the compactness of the support of $f_\ep$ with respect to $\bv$, all ``boundary terms'' arising in the course of the partial integrations are annihilated, we obtain the following weak formulation of the system \eqref{e:1}, \eqref{e:2}, where ``$:$'' denotes the scalar product in $\mR^{3 \times 3}$:
\begin{subequations}\label{mn:wk}
\begin{align}
\nonumber
&\int_{\mTx \times \mRp}  \fe(t,\bx,\bv)\, g(t,\bx,\bv) \td^3\bx \td^3\bv - \int_{\mTx \times \mRp} \mr\fe(\bx,\bv)\, g(0,\bx,\bv) \td^3\bx \td^3\bv -\int_0^t \int_{\mTx \times \mRp} \fe \pt g \td^3\bx \td^3\bv \td s\\
\label{f:mn:wk}
&\qquad=\int_0^t \int_{\mTx \times \mRp} (\bv  \fe)\cdot \nabla_{\bx} g-\big(( \uee-\bv)\times \Bee \fe\big)\cnp g \td^3\bx \td^3\bv \td s   \\
\intertext{for all $g \in\bC^1_c([-1,T+1]\times\mTx\times\mRp)$;}
\nonumber
&\int_\mTx \ue(t,\bx) \ccdot \vv(t,\bx) \td^3\bx -\int_\mTx \mr \ue(\bx) \ccdot \vv(0,\bx)\td^3\bx - \int_0^t \int_\mTx \ue\ccdot\pt\vv \td^3\bx \td s\\
\nonumber
&\qquad =\int_0^t \int_\mTx \big( \uee\otimes\ue- \Bee\otimes\Be\big) : \nabla_{\bx}\vv \td^3\bx \td s - \int_0^t \int_\mTx \nabla_{\bx}\ue : \nabla_{\bx}\vv \td^3\bx \td s \nonumber\\
&\qquad\qquad+ \int_0^t \int_\mTx \big(\ue\times \Bee\big)\Big(\intp{\fe} \Big)\ccdot\vv { +} \Big(  \Bee\times\intp{\bv\fe} \Big)^{\langle\ep \rangle} \ccdot\vv \td^3\bx \td s\label{u:mn:wk}\\
\intertext{for all $\vv \in\bC^1([-1,T+1]\times\mTx;\mathbb{R}^3)$ satisfying $\nabla_{\bx}\ccdot \vv=0$; and}
&\int_\mTx \Be(t,\bx) \ccdot \vv(t,\bx) \td^3\bx -\int_\mTx \mr{\mathbf{B}}_\varepsilon(\bx) \ccdot \vv(0,\bx) \td^3\bx - \int_0^t \int_\mTx  \Be \ccdot\pt\vv \td^3\bx \td s\nonumber\\
&\qquad=\int_0^t  \int_\mTx\big( \uee\otimes\Be- \Bee\otimes\ue\big) : \nabla_{\bx}\vv \td^3\bx \td s- \int_0^t \int_\mTx\nabla_{\bx}\Be : \nabla_{\bx}\vv \td^3\bx \td s \label{B:mn:wk}
\end{align}
for all $\vv \in\bC^1([-1,T+1]\times\mTx;\mathbb{R}^3)$ satisfying $\nabla_{\bx}\ccdot \vv=0$.
\end{subequations}
\medskip

As a classical solution, whose existence is guaranteed by Theorem \ref{thm:mo}, is thereby automatically a weak solution (in the
sense of \eqref{f:mn:wk}--\eqref{B:mn:wk}), we directly deduce the existence of a triple $(f_\ep,\ue,\Be)$ satisfying  \eqref{f:mn:wk}--\eqref{B:mn:wk} for the initial data \eqref{e:in:def} under consideration.

Next, we summarize, without proof, some standard properties of mollifiers, which will be extensively used in the course of the discussion that follows.

\begin{lemma}\label{lem:mo}
Suppose that $\Omega$ is one of $~\mTx$, $\mRp$ or $\mTx \times \mRp$,  and let $\theta^\ep$ denote one of $\theta^\ep_{\bx}$, $\theta^\ep_{\bv}$ or $\theta^\ep_{\bx} \theta^\ep_{\bv}$, respectively, with $\ep>0$. Let $\bX$ denote one of $L^r(\Omega)$ or $L^s ((0,T);  L^r(\Omega))$, where $r,s\in[1,\infty)$ and $T>0$. Then,
\begin{subequations}
\label{mo:prop:conv}
\begin{align}
\label{mo:prop:conv:a}&\text{for any  $w\in  \bX$,   }\lim_{\ep\to0}\|w* \theta^{\ep}- w\|_\bX=0\text{   and }\|w* \theta^{\ep}\|_\bX\le\|w\|_\bX;\\
\label{mo:prop:conv:b}
&\begin{aligned}
& \text{for  any  $w\in  \bX$, $\{w_n\}_{n \geq 1}\subset\bX$, and any sequence $\{\ep_n\}_{n \geq 1}$ of positive real numbers, }\\
&\qquad\quad\text{if $\lim_{n\to\infty}\ep_n=0$ and $\lim_{n\to\infty}\|w_n- w\|_\bX=0$, then  }\lim_{n\to\infty}\|w_n* \theta^{\ep_n}- w\|_\bX=0.
\end{aligned}
\end{align}
\end{subequations}
\end{lemma}
Here, \eqref{mo:prop:conv:b} follows from \eqref{mo:prop:conv:a} and the triangle inequality
$$
\|w_n* \theta^{\ep_n}- w\|_\bX\le \|w_n* \theta^{\ep_n}- w* \theta^{\ep_n}\|_\bX+\|w* \theta^{\ep_n}- w\|_\bX.
$$
It then   follows  that
\be
\label{conv:mr}
\lim_{\ep\to0}\big\|(\mr{\mathbf{U}}_\varepsilon,\mr{\mathbf{B}}_\varepsilon)-(\mr\vu,\mr\vB)\big\|_{L^2(\mTx)}=0\quad\text{and} \quad
\text{for }r\in[1,\infty),\;\;\lim_{\ep\to0}\|\mr \fe- \mr f\|_{L^r(\mTx\times\mRp)}=0.
\ee
Also, $\|\mr\fe\|_{L^r(\mTx \times \mRp)} \leq \|\mr f\|_{L^r(\mTx \times \mRp)}$, for all $r\in[1,\infty]$, including the $L^\infty$ norm. Thus, thanks to the  $L^r$ invariance property \eqref{fn:Lr:mollify}, we have the uniform $L^r$ bounds
\be\label{fe:Lr:uni}  \|\fe(t)\|_{L^r(\mTx \times \mRp)}\le \|\mr f\|_{L^r(\mTx \times \mRp)}\quad\text{for all}\;\; t\in[0,T],\;\; r\in[1,\infty].\ee
\smallskip

Concerning the initial energy of particles $\Epar{\mr\fe}$, by shifting the mollifier under the integral sign, we have that
\[
\Epar {\mr\fe} =\int_{\mTx\times\mRp}{1\over2} \big(|\bv|^2*(\thx\,\thp)\big ) \,\big(\mr f\ccdot\mychi{|\bv|\le1/\ep} \big)\td^3\bv\td^3\bx\le \int_{\mTx\times\mRp}{1\over2} \big(|\bv|^2*(\thx\,\thp)\big ) \, \mr f \td^3\bv\td^3\bx.
\]
Since $\thp$ is an even function with unit integral over $\mRp$,
we have
\[
|\bv|^2*(\thx\thp) =\int_{\mR^3} \int_{\mT^3} |\bv-\bw|^2\,\thp(\bw)\,\thx(\by)\td^3 \by\td^3 \bw=\int_{\mR^3} (|\bv|^2+|\bw|^2)\,\thp(\bw)\td^3 \bw= |\bv|^2+C_\theta\ep^2,
\]
where $C_\theta:=\int_{\mR^3}|\bq|^2\,\theta(\bq)\td^3 \bq<1$ (c.f. the definition of   $\theta(\cdot) $ at the start of Section \ref{sec:mo}). Hence we deduce that
\[
 \Epar {\mr\fe} \le \Epar {\mr f}+\ep^2|\mr f|_{L^1(\mTx \times \mRp)}.
\]
Combining this with \eqref{mo:prop:conv:a} and recalling the definition of $ \Ettinc$  we have that
\be
\label{IC:en:bound}
\Ettinc[\mr \fe,\mr{\mathbf{U}}_\varepsilon,\mr{\mathbf{B}}_\varepsilon]\le \Ettinc[\mr f,\mr\vu,\mr\vB]+\ep^2|\mr f|_{L^1(\mTx \times \mRp)}.
\ee
By further considering the energy equality \eqref{energy:mollify} and the uniform $L^\infty$ bound in \eqref{fe:Lr:uni}, we obtain
\be
\begin{aligned}
\label{uni:bound:mn}
\Ettinc[\fe,\ue,\Be](t)  +\|(\nabla_{\bx} \ue,\nabla_{\bx}\Be)\|_{L^2((0,T); L^{2}(\mTx))} + \|\fe(t)\|_{\bC(\mTx \times \mRp)}\le \Fo,
\end{aligned}
\ee
for every $t \in [0,T]$ and $\ep\in(0,1]$, where
\[
\Fo:=\Ettinc[\mr f,\mr\vu,\mr\vB]+|\mr f|_{L^\infty(\mTx \times \mRp)}+|\mr f|_{L^1(\mTx \times \mRp)}.
\]
Then, by Proposition \ref{prop:Lr} we also have that
\be
\label{uni:bound:mo}
\Big\|\intp{\fe\,} \Big\|_{\bC([0,T];L^{5\over 3}(\mTx))}+\Big\|\intp{\bv\,\fe\,} \Big\|_{\bC([0,T];L^{5\over 4}(\mTx))}\le C\Fo.
\ee
Here and henceforth $C$ will signify a generic positive constant that is independent of   $\ep$.

\subsection{Time regularity and compactness of the sequence $\{(\ue,\Be)\}_{\ep>0}$}
We would like to apply the Aubin--Lions--Simon compactness result stated in Theorem \ref{embedding:lemma} to the sequence $\{(\ue,\Be)\}_{\ep>0}$ to deduce its strong convergence in a suitable norm, and to this end an $\ep$-uniform bound on $\{(\pt\ue,\pt\Be)\}_{\ep>0}$ is needed. Since the left-hand sides of \eqref{u:mn:wk} and \eqref{B:mn:wk} are equal to, respectively
\[
\int_0^t \int_{\mTx}  \pt \ue\ccdot \vv \td^3\bx \td s  \quad\text{ and } \quad \int_0^t \int_{\mTx} \pt \Be\ccdot \vv \td^3\bx \td s,
\]
we will focus on bounding the right-hand sides of \eqref{u:mn:wk} and \eqref{B:mn:wk}, and in particular the trilinear and quadrilinear terms. Note that an $\ep$-uniform bound on $\{\pt \fe\}_{\ep>0}$ is not sought here, since the weak and weak* compactness of  the sequence $\{\fe\}_{\ep>0}$ will suffice for our purposes.
\begin{enumerate}[(i)]
\item To estimate the first integrals on the right-hand sides of \eqref{u:mn:wk}, \eqref{B:mn:wk}, we employ Ladyzhenskaya's inequality (which is a special case of the, more general, Gagliardo--Nirenberg inequality) to deduce that
\[
\|\ue(t)\|_{L^4(\mTx)}\le C \|\ue(t)\|_{L^2(\mTx)}^{1\over4}\|\nabla_{\bx}\ue(t)\|_{L^2(\mTx)}^{3\over4}+C \|\ue(t)\|_{L^2(\mTx)},\;\;\;\text{and likewise for }\Be(t),
\]
for any $t \in [0,T]$.
Therefore, by noting the uniform energy bound \eqref{uni:bound:mn}, we deduce that
\[
\big\|(\ue,\Be)\big\|_{L^{8\over3}((0,T); L^{4}(\mTx))}\le C_T \Fo.
\]
Here and henceforth $C_T$ will signify a generic positive constant that may depend on $T$ but is independent of the mollification parameter $\ep$.

Hence, by H\"older's inequality (applied twice) and the fact that mollification does not increase Sobolev norms, we have that
\[
\begin{aligned}
&\Big|\text{the first integral on the right-hand side of \eqref{u:mn:wk} and \eqref{B:mn:wk}}\Big|
\\&\qquad\le C\int_0^T\|(\ue,\Be)\|^2_{L^4(\mTx)}\|\nabla_{\bx} \vv\|_{L^2(\mTx)}\td t\\
&\qquad \le C\Big(\int_0^T\|(\ue,\Be)\|^{2\ccdot{4\over 3}}_{L^4(\mTx)}\td t\Big)^{3\over 4}\Big(\int_0^T\|\nabla_{\bx} \vv\|^4_{L^2(\mTx)}\td t\Big)^{1\over 4}\\
&\qquad \le C_T \,\Fo^2\,\|  \vv\|_{L^4((0,T);H^1(\mTx))}.
\end{aligned}
\]
\item The second integral on the right-hand side of \eqref{u:mn:wk} and \eqref{B:mn:wk} is bounded by
$C\Fo\|  \vv\|_{L^{2}((0,T);H^1(\mTx))}$.

\item It remains to bound the last integral of \eqref{u:mn:wk}. We invoke the Gagliardo--Nirenberg inequality,
\[
\|\ue\|_{L^{5}(\mTx)}\le C\|\ue\|_{L^2(\mTx)}^{1\over10}\| \nabla_{\bx}\ue\|_{H^1(\mTx)}^{9\over10}+C\|\ue\|_{L^2(\mTx)},\quad\text{and likewise for }\Be,
\]
for any $t \in [0,T]$, and combine it with the uniform energy bound \eqref{uni:bound:mn} to obtain
\[
\big\|(\ue,\Be)\big\|_{L^{20\over9}((0,T); L^{5}(\mTx) )}\le C_T\Fo.
\]
Combining this with the uniform bounds on the moments in \eqref{uni:bound:mo} and applying H\"older's inequality (twice) together with  the fact that mollification does not increase Sobolev norms, we get that
\[
\begin{aligned}
&\Big|\text{the last integral  of \eqref{u:mn:wk}}\Big| \\
&\qquad\le C\,\Fo\,\int_0^T\|\ue\|_{L^{5}(\mTx)}\,\|\Be\|_{L^5(\mTx)}\|  \vv\|_{\bC(\mTx)}\td t + C\,\Fo\,\int_0^T \|\Be\|_{L^5(\mTx)}\|  \vv\|_{\bC(\mTx)}\td t\\
&\qquad\le C\,\Fo\,\big\| \ue \big\|_{L^{20\over9}((0,T); L^{5}(\mTx))} \,\big\| \Be \big\|_{L^{20\over9}((0,T); L^{5}(\mTx))} \|  { \vv}\|_{L^{10}((0,T);\bC(\mTx))}\\
&\qquad\qquad + C\,\Fo\,\big\| \Be \big\|_{L^{20\over9}((0,T); L^5(\mTx))}\|  \vv\|_{L^{20\over11}((0,T);\bC(\mTx))}\\
&\qquad\le C_T\,(\Fo^2+\Fo^3)\,\|  \vv\|_{L^{10}((0,T);H^2(\mTx))}\qquad\text{ (by Sobolev inequalities applied to $\vv$)}.
\end{aligned}
\]
\end{enumerate}

\smallskip

By combining the bounds established in (i), (ii), (iii) with \eqref{u:mn:wk}, \eqref{B:mn:wk} we deduce that,
\[
\Big|\int_0^T\int_\mTx \pt\ue \ccdot \vv \td^3\bx \td s \Big|+\Big|\int_0^T\int_\mTx \pt\Be\ccdot \vv \td^3\bx \td s\Big|\le C_T\,(\Fo+\Fo^2+\Fo^3)\,\|  \vv\|_{L^{10}((0,T);H^2(\mTx))}.
\]
This estimate now implies that, for fixed $T>0$, any subsequence of $\{(\pt\ue,\pt\Be)\}_{\ep>0}$ is uniformly bounded in $L^{10\over 9 }\left((0,T);\, \big(H^{2}(\mTx;\mR^6)\big)^*\right)$, where   $\big(H^{2}(\mTx;\mR^6)\big)^*$ is the dual space of $H^2(\mTx;\mR^6)$. Together with the uniform energy bound \eqref{uni:bound:mn}, this bound allows us to apply Theorem \ref{embedding:lemma} to the sequence $\{(\ue,\Be)\}_{\ep>0}$ with
\[r=2,\;\; s={10\over 9 },\;\; \bX_1=H^1(\mTx;\mR^6),\;\;\bX_0=L^5(\mTx;\mR^6),\;\;\bX_{-1}=\big(H^{2}(\mTx;\mR^6)\big)^*.\]
Indeed, using  the $L^2(\mTx)$ inner product for the duality pairing, we have that the continuous embedding $H^2(\mTx)\hookrightarrow L^{5\over4}(\mTx)$ implies the continuous embedding    $ L^{5}(\mTx)=(L^{5\over4}(\mTx))^* \hookrightarrow (H^2(\mTx))^*$,
and therefore  we   have $\bX_0$   continuously embedded in $\bX_{-1}$. Also, by the Rellich--Kondrashov theorem, we have  compact embedding of $\bX_1$ into $\bX_0$. Therefore, the hypotheses of Theorem \ref{embedding:lemma} are satisfied.

\begin{lemma}\label{lm:comp:uB}
Let $(\mr\vu,\mr\vB)\in L^2(\mTx;\mR^6)$, $ \mr f \in L^\infty(\mTx\times\mRp)\cap L^1(\mTx\times\mRp)$, $T>0$, and consider, for   $t \in (0,T]$, the family of solutions to \eqref{e:1}, \eqref{e:2} with mollified initial data \eqref{e:in:def}. Then, there exist a sequence of positive real numbers $\{\ep_n\}_{n\geq 1}$ satisfying $\displaystyle\lim_{n\to\infty}\ep_n=0$ and a limit solution $(\ub, \Bb)\in L^2((0,T);H^1(\mTx;\mR^6))\cap  L^\infty((0,T);L^2(\mTx;\mR^6))$, such that
\[
\begin{aligned}
(\uen, \Ben)\to ( \ub, \Bb)\quad&\text{strongly in }\; L^2((0,T);L^5(\mTx;\mR^6)),\;\;\;\text{weakly in }\; L^2((0,T);H^1(\mTx;\mR^6)),\\
&\text{and weak* in }\; L^\infty((0,T);L^2(\mTx;\mR^6))\quad\text{as }\;\;n\to \infty.
\end{aligned}
\]
\end{lemma}
Here, the weak and weak* convergence results are direct consequences of the uniform energy bound \eqref{uni:bound:mn},   the reflexivity of   $L^2((0,T);H^1(\mTx;\mR^6))$ and the Banach--Alaoglu theorem. The strong convergence result in $L^2((0,T);L^5(\mTx;\mR^6))$ asserted in Lemma \ref{lm:comp:uB} will play an important role later on, in passing to the limit in the trilinear and quadrilinear terms  in \eqref{u:mn:wk} that involve the moments of $\fe$.

\subsection{Weak* convergence of the sequence $\{\fe\}_{\ep>0}$ and its moments}

The aim of this section is to establish the following lemma, concerning weak* convergence of the sequence $\{\fe\}_{\ep>0}$ and of its moments.
\begin{lemma}
\label{lm:comp:f}
Under the hypotheses of Lemma \ref{lm:comp:uB}, there exist a subsequence of $\{\ep_n\}_{n \geq 1}$, still denoted by $\{\ep_n\}_{n \geq 1}$ satisfying $\displaystyle\lim_{n\to\infty}\ep_n=0$, and a limit function $\fb=\fb(t,\bx,\bv)$, such that, as $n\to\infty$,
\begin{alignat}{2}
\label{wk:conv:Linf}
\,\fen &\to \fb &&\quad \mbox{weak* in $L^\infty((0,T)\times\mTx\times\mRp)$},\\
\label{wk:f:Lr:pos}
f &\ge0 &&\quad \mbox{everywhere on $[0,T]\times\mTx\times\mRp$},\\
\label{wk:conv:Lr}
\,\fen &\to \fb &&\quad \mbox{weak* in $L^\infty((0,T);L^r(\mTx \times \mRp))$ for all $r\in(1,\infty)$},\\
\label{conv:mo:pf}
\intp{\bv\,\fen (\cdot,\cdot,\bv)\,}&\to \intp{\bv\,\fb(\cdot,\cdot,\bv)\,}&&\quad \text{weak* in }  L^{\infty}((0,T); L^{5\over4}(\mTx)),\\
\label{conv:mo:f}
\intp{\fen (\cdot,\cdot,\bv)\,}&\to \intp{\fb(\cdot,\cdot,\bv)\,}&&\quad\text{weak* in }  L^{\infty}((0,T); L^{5\over3}(\mTx)),\\
\label{fb:Lr:le}
\|\fb\|_{L^\infty((0,T);L^r(\mTx \times \mRp))}&\le\|\mr f\|_{ L^r(\mTx \times \mRp)} &&\quad\text{ for all $r\in[1,\infty]$}.
\end{alignat}
\end{lemma}
\begin{proof}
For the sake of simplicity of the notation, repeatedly extracted subsequences involved in the proof will all be denoted by $\{\fen\}_{n \geq 1}$. As before, $\mychi{|\bv|\le N}$ will denote the cut-off function taking the value 1 in the   ball of radius $N$ in $\mRp$ centred at the origin, and equal to 0 otherwise; let $\mychi{|\bv|> N}=1-\mychi{|\bv|\le N}$.

The uniform $L^\infty$ bound on $\{\fe\}_{\ep>0}$ established in \eqref{fe:Lr:uni}  and the Banach--Alaoglou theorem imply \eqref{wk:conv:Linf}. Then,
the nonnegativity of $f$ on $[0,T] \times \mTx \times \mRp$ stated in \eqref{wk:f:Lr:pos} follows from the nonnegativity  of the continuous functions $\fen$ by Lemma \ref{thm:S:bound}. Indeed, for any   nonnegative   function $\eta\in  L^1((0,T)\times \mTx \times \mRp)$, the weak* convergence \eqref{wk:conv:Linf} implies
\[
\int_0^T\int_{\mTx \times \mRp} f \, \eta  \td^3\bx \td^3\bv \td t =\lim_{n\to\infty}\int_0^T\int_{\mTx \times \mRp} \fen \, \eta  \td^3\bx \td^3\bv \td t\geq 0.
\]
By choosing $\eta=\mychi{f\le-{1\over N}}\ccdot\mychi{|\bv|<N}$ for any $N>0$,
we then deduce that $f \geq 0$ a.e. on $(0,T)\times \mTx \times \mRp$, and we can then modify $f$ on a subset   of $[0,T]\times \mTx \times \mRp$ with zero Lebesgue measure to ensure its nonnegativity everywhere.  The already proven weak* convergence result  \eqref{wk:conv:Linf} is not affected by such an alteration on a set of zero Lebesgue measure.

The uniform $L^r$ bound on $\{\fe\}_{\ep>0}$ for $r\in(1,\infty)$ in \eqref{fe:Lr:uni}  and the Banach--Alaoglou theorem imply \eqref{wk:conv:Lr}, but only for a fixed $r$. Proving that there exists a subsequence   $\{\fen\}_{n\geq 1}$ that \textit{simultaneously} weak* converges in $L^\infty((0,T);L^r(\mTx \times \mRp))$ for all $r\in(1,\infty)$ requires a subtle argument because of the lack of compactness for $r=1$. We begin by finding a sequence of nested subsequences (starting with the one used for \eqref{wk:conv:Linf}):  $\{\fen\}_{n \geq 1}\supset S_1\supset S_2\supset S_3\supset \cdots$  such that the subsequence $S_n$ makes \eqref{wk:conv:Lr} true for $r=1+2^k$ with $k=(-1)^n\lfloor {n\over2}\rfloor$. By a test-function argument, all these (countably many) weak* limits can be taken to be the same $f$ as in \eqref{wk:conv:Lr}. Then, by a diagonal argument, we construct the subsequence $\{\fen\}_{n \geq 1}$ such that $\fen$ is the $n$-th
element of $S_n$, which makes \eqref{wk:conv:Lr} simultaneously true for all $r=1+2^k$ where $k\in\mZ$.

By the uniform $L^r$ bound \eqref{fe:Lr:uni} and the weak* lower-semicontinuity of  the norm of any Banach space, we have $\|\fb\|_{L^\infty((0,T);L^r(\mTx \times \mRp))}\le\|\mr f\|_{ L^r(\mTx \times \mRp)}$ for all $r=1+2^k$ where $k\in\mZ$. Then, we use H\"older's inequality to interpolate these $L^{1+2^k}$ norms and obtain
\be\label{f:Lr:mid}
\|(f,\fen)\|_{L^\infty((0,T);L^r(\mTx \times \mRp))^2}\le 2\big(  \|\mr f\|_{L^\infty(\mTx \times \mRp)}+ \|\mr f\|_{L^1(\mTx \times \mRp)}\big) =: C_0\quad\text{for all}\;\;\;r\in(1,\infty).
\ee

Next, for  any  $r\in(1+2^{k-1},1+2^{k})$, consider the conjugate,  $r':=r/( r-1)$, and choose any $g$ in the function space
$L^1((0,T);L^{r'}(\mTx \times \mRp))$, which is the \textit{predual} of the space $L^\infty((0,T);L^r(\mTx \times \mRp))$.
The size of $g$ for large values of $|\bv|$ can then be made sufficiently small, in the sense that
\be\label{lim:g:N}
\lim_{N\to\infty} \|g \,\mychi{|\bv|> N} \|_{L^1((0,T);L^{r'}(\mTx \times \mRp))}=0.
\ee
By fixing  $N$, applying   H\"older's inequality and using the uniform estimate \eqref{f:Lr:mid}, we have
\begin{align}
\label{ge:N:r}
&\left|\int_0^T  \int_{\mTx \times \mRp} (\fen-\fb)\, g\, \mychi{|\bv|> N} \td^3\bx \td^3\bv \td t\right|
\le C_0\, \|g \,\mychi{|\bv|> N} \|_{L^1((0,T);L^{r'}(\mTx \times \mRp))}.
\end{align}
Also, since $r'>(1+2^k)/(1+ 2^k-1)= 1 + 2^{-k}$ and $g\,  \mychi{|\bv|\le N}\in L^1((0,T_;L^{r'}(\mTx \times \mRp))$ is compactly supported, we must have
$g\,\mychi{|\bv|\le N}\in{L^1((0,T);L^{1+2^{-k}}(\mTx \times \mRp))}$,
whose dual   is ${L^\infty((0,T);L^{1+2^{k}}(\mTx \times \mRp))}$.
Recall  that \eqref{wk:conv:Lr} for $r=1+2^k$ has been established, and thus
\[
\lim_{n\to\infty} \int_0^T\int_{\mTx \times \mRp}  (\fen-\fb)\, g\, \mychi{|\bv|\le N}  \td^3\bx \td^3\bv \td t=0.
\]
Combining this with the uniform limit \eqref{lim:g:N} and the uniform estimate \eqref{ge:N:r},  we have that
\[
\lim_{n\to\infty} \int_0^T\int_{\mTx \times \mRp}  (\fen-\fb)\, g   \td^3\bx \td^3\bv \td t=0\quad\text{for any}\;\;\;g\in L^1((0,T);L^{r'}(\mTx \times \mRp))\,
\]
and therefore the  subsequence $\{\fen\}_{n \geq 1}$ and the weak* limit $f$ we have constructed so far   make \eqref{wk:conv:Lr}   true for any  $r\in(1+2^{k-1},1+2^{k})$ for all $k \geq 1$. The proof of  \eqref{wk:conv:Lr} is therefore complete.

To show \eqref{conv:mo:pf}, we fix any $\phi \in L^\infty((0,T) \times \mTx)$.  Then, we have $v_1^+\mychi{|\bv|\le N}\phi\in L^{1}((0,T) \times \mTx \times \mRp)$ for any $N>0$, where $v_1$ is the first coordinate of $\bv=(v_1,v_2,v_3)$ and $v_1^+:=\max\{v_1,0\}$. We apply \eqref{wk:conv:Linf} to deduce that
\begin{align}\label{eq:eq1}
\int_0^T\int_{\mTx \times \mRp} v_1^+\,\fb\mychi{|\bv|\le N}\,\phi(t,\bx) \td^3\bx \td^3\bv \td t-\lim_{n\to\infty}\int_0^T\int_{\mTx \times \mRp}  v_1^+\, \fen    \,\mychi{|\bv|\le N}\,\phi(t,\bx) \td^3\bx \td^3\bv \td t= 0.
\end{align}
By the uniform energy bound \eqref{uni:bound:mn} and letting $  K:= { \textnormal{ess\,}\sup}_{\,n \in \mathbb{N},\,(t,\bx) \in [0,T] \times \mTx}\big\{ |\phi(t,\bx)|,\,\Epar{\fen }(t)\big\}<\infty$, we obtain the following estimate concerning large values of $|\bv|$:
\begin{align}\label{eq:eq2}
\Big|\int_0^T\int_{\mTx \times \mRp} v_1^+\,{\fen \,} \mychi{|\bv|>N} \phi(t,\bx)\td^3\bx \td^3\bv \td t \Big|\le  \Big|\int_0^T\int_{\mTx \times \mRp}
{|\bv|^2\over N}\,{\fen \,} K \td^3\bx \td^3\bv \td t\Big|    \le {2\,TK^2\over N}.
\end{align}
As the expression on the left-hand side of \eqref{eq:eq2} is a bounded sequence in $\mathbb{R}$ (with respect to $n$), it has a convergent subsequence (not indicated). Thus, taking the limit $n \rightarrow \infty$ in \eqref{eq:eq2} over this subsequence and subtracting the resulting inequality from \eqref{eq:eq1} we deduce that
\[
-{2\,TK^2\over N}\le \int_0^T\int_{\mTx \times \mRp} v_1^+\, \fb \mychi{|\bv|\le N}\,\phi(t,\bx)  \td^3\bx \td^3\bv \td t-\lim_{n\to\infty} \int_0^T\int_{\mTx \times \mRp}  v_1^+\, \fen  \,\phi(t,\bx) \td^3\bx \td^3\bv \td t  \le {2\,TK^2\over N}.
\]
Next, we shall pass to the limit $N\to\infty$ in this inequality; to this end, we shall suppose that $\phi(t,\bx)\ge 0$. Thanks to the nonnegativity of $f$, we can apply the monotone convergence theorem to the limit
\[
\lim_{N\to\infty}\int_0^T\int_{\mTx \times \mRp}   v_1^+\, \fb \mychi{|\bv|\le N}\,\phi(t,\bx)  \td^3\bx \td^3\bv \td t =
\int_0^T\int_{\mTx \times \mRp}   v_1^+\, \fb  \,\phi(t,\bx)  \td^3\bx \td^3\bv \td t
\]
and hence combine the last two estimates/limits to deduce that, for any $\phi$ such that
$0\le\phi \in L^\infty((0,T)\times\mTx)$,
\[
\int_0^T\int_{\mTx \times \mRp} v_1^+\,\fb\phi(t,\bx) \td^3\bx \td^3\bv \td t=\lim_{n\to\infty}\int_0^T\int_{\mTx \times \mRp}  v_1^+\, \fen  \phi(t,\bx) \td^3\bx \td^3\bv \td t.
\]
Since the uniform estimate \eqref{uni:bound:mo} guarantees that both sides are finite and hence all integrands (which are nonnegative) are Lebesgue-integrable, we can apply Fubini's theorem to deduce that
\[
\int_0^T\int_{\mTx} \Big(\int_\mRp v_1^+\,\fb \td^3\bv \Big) \phi(t,\bx)\td^3\bx \td t=\lim_{n\to\infty}\int_0^T\int_{\mTx} \Big(\int_\mRp  v_1^+ \, \fen \td^3\bv \Big)  \,  \phi(t,\bx) \td^3\bx \td t
\]
for any $\phi\geq 0$ such that $\phi \in L^\infty((0,T)\times \mTx)$.

We then repeat the same procedure for any $0\ge\phi \in L^\infty((0,T) \times \mTx \times \mRp)$, following which we repeat the reasoning for $v_1^-$, $v_2$, $v_3$, to finally deduce that
\[
\int_0^T\int_{\mTx} \Big(\int_\mRp \bv\,\fb \td^3\bv \Big) \phi(t,\bx)\td^3\bx \td t=\lim_{n\to\infty}\int_0^T\int_{\mTx} \Big(\int_\mRp  \bv \, \fen \td^3\bv \Big)  \,  \phi(t,\bx) \td^3\bx \td t
\]
for all $\phi \in L^\infty((0,T)\times \mTx)$.

On the other hand, thanks to the uniform estimate \eqref{uni:bound:mo} and the Banach--Alaoglou theorem, there exists an $M_1 \in L^\infty((0,T);L^{5\over 4}(\mTx))$ such that, upon extraction of a subsequence (not indicated),
\[\intp{\bv\,\fen(\cdot,\cdot,\bv)\, }\to M_1\text{ \;\; weak* in \;\; }  L^{\infty}((0,T); L^{5\over4}(\mTx))\quad\text{as }\;\;n\to\infty.\]
Combining these two limits we deduce that
\[
\int_0^T\int_\mTx \Big( \intp{\bv\,\fb(t,\bx,\bv)\,}\Big) \,\phi(t,\bx) \td^3\bx \td t = \int_0^T\int_\mTx M_1(t,\bx)\, \phi(t,\bx) \td^3\bx \td t
\;\;\text{ for all }\;\phi \in L^\infty((0,T)\times \mTx),
\]
which directly implies (for example by du Bois-Reymond's lemma) that
\[
\intp{\bv\,\fb(t,\bx,\bv)\,}=M_1(t,\bx) \quad \mbox{a.e. on $(0,T) \times \mTx$},
\]
and thus \eqref{conv:mo:pf} has been proved. The proof of \eqref{conv:mo:f} proceeds analogously and is therefore omitted.

Finally, the bound \eqref{fb:Lr:le} for $r\in(1,\infty]$ follows from the uniform $L^r$ bound  \eqref{fe:Lr:uni}, the
weak* convergence  \eqref{wk:conv:Linf} and \eqref{wk:conv:Lr}, and the weak* lower-semicontinuity of the norm of a Banach space.
Note that covering the case of $r=\infty$ requires particular care. In fact, for any nonnegative, measurable function $f$ defined on the set $(0,T)\times\mTx \times \mRp$, we have  $\|\fb\|_{L^\infty((0,T)\times\mTx \times \mRp)}=\|\fb\|_{L^\infty((0,T);L^\infty(\mTx \times \mRp))}$ thanks to the following argument: with $A_1:=\|\fb\|_{L^\infty((0,T)\times\mTx \times \mRp)}$ and $A_2:=\|\fb\|_{L^\infty((0,T);L^\infty(\mTx \times \mRp))}$, we deduce that $A_1\le A_2$ by applying Fubini's theorem  to   $\int_0^T\int_{|\bv|<N } \!  \int_{\mTx}\mychi{\fb\ge A_1-\ep}\td^3\bx \td^3\bv \td t$ for arbitrary $\ep>0$ and sufficiently large $N>0$, and we then prove that $A_1\ge A_2$ by applying Fubini's theorem  to   $\int_0^T \int_{\mTx \times \mRp}\mychi{\fb\ge A_1}\td^3\bx \td^3\bv \td t$.

It remains to prove \eqref{fb:Lr:le} for $r=1$, which does not, in fact, rely on \eqref{wk:conv:Lr} but on \eqref{conv:mo:f}.
Consider a univariate integrable function $g \in L^1(0,T)$. Hence, automatically, $g \in L^{1}((0,T);L^{5\over2}(\mTx))$, whose dual space is $L^\infty((0,T);L^{5\over3}(\mTx))$. Therefore, by \eqref{conv:mo:f},
\[
\int_0^T\int_\mTx \Big( \intp{\fb(t,\bx,\bv)\,}\Big) \,g(t) \td^3\bx \td t
= \lim_{n\to\infty}\int_0^T\int_{\mTx} \Big( \intp{\fen(t,\bx,\bv)\,}\Big) \,g(t) \td^3\bx \td t,
\]
where all integrands are Lebesgue integrable. We then apply Fubini's theorem and use the fact that both $f$ and $\fen$ are nonnegative to deduce that
\[
\int_0^T\|f(t)\|_{L^1(\mTx \times \mRp)}\,g(t)\td t\le\liminf_{n\to\infty}\int_0^T\|\fen(t)\|_{L^1(\mTx \times \mRp)}\,g(t)\td t.
\]
By the uniform $L^r$ bound \eqref{fe:Lr:uni}, we have that the right-hand side is dominated by
$\|\mr f\|_{L^1(\mTx \times \mRp)}\|g\|_{L^1(0,T)}$. Hence, by taking the supremum over all $g$ such
that $\|g\|_{L^1(0,T)}=1$ gives
\[\Big(\|f(t)\|_{L^1(\mTx \times \mRp)}\Big)_{(L^1(0,T))^*} \le \|\mr f\|_{L^1(\mTx \times \mRp)}.
\]
Since the norm in the dual space $(L^1(0,T))^*$ is identical to the $L^\infty(0,T)$ norm,
we have proved \eqref{fb:Lr:le} for $r=1$ as well. That completes the proof.
\end{proof}

With these convergence results in place, we are now ready to pass to the limit $\ep \rightarrow 0$ in \eqref{mn:wk} to prove the main result of the paper; this will be the subject of the next section.

\subsection{The limit solves the weak form of the PDE}
It remains to prove that the limits identified in Lemmas \ref{lm:comp:uB}, \ref{lm:comp:f} satisfy the weak form of the original (nonmollified) PDE \eqref{eq}, in a sense that will be made precise in the next definition.

\begin{definition}
\label{def:wk:inc}All functions in this definition are understood to be $2\pi$-periodic with respect to the independent variable $\bx$. Suppose that the initial datum $\mr f \in L^1(\mTx\times\mRp)\cap L^\infty(\mTx\times\mRp)$ is such that $\Epar{\mr f}$ is
finite, and consider divergence-free initial data $(\mr\vu,\mr\vB)\in \incL(\mTx;\mR^6)$ and $T>0$. We call
\begin{subequations}\label{def:wk:reg:int}
\be( \vu, \vB)\in L^\infty((0,T);\incL(\mTx;\mR^6))\cap L^2((0,T);H^1(\mTx;\mR^6)),
\ee
with $\nabla_{\bx} \cdot \vu=0$, $\nabla_{\bx} \cdot \vB=0$ a.e. on $(0,T) \times \mTx \times \mRp$,  and
\be
f\in  L^\infty\big((0,T); L^1(\mTx \times \mRp)\cap L^\infty(\mTx \times \mRp)),
\ee
\end{subequations}
a \textit{weak solution} to the hybrid incompressible Vlasov-MHD system \eqref{eq}, if, for \emph{every} $t \in [0,T]$,  the following are true:
\begin{subequations}
\label{wk}
\begin{align}
\nonumber
&\int_{\mTx \times \mRp}f(t,\bx,\bv)\,g(t,\bx,\bv) \td^3\bx \td^3\bv - \int_{\mTx \times \mRp} \mr f(\bx,\bv) \,  g(0,\bx,\bv) \td^3\bx \td^3\bv - \int_0^t \int_{\mTx \times \mRp} f\,\pt g
\td^3\bx \td^3\bv \td s\\
\label{f:wk}
&=\int_0^t \int_{\mTx \times \mRp}(\bv  f)\cdot \nabla_{\bx} g-\big(( \vu -\bv)\times \vB f\big)\cnp g \td^3\bx \td^3\bv \td s\\
\intertext{for any compactly supported, scalar-valued test function $g\in\bC^1_c([-1,T+1]\times\mTx\times\mRp)$;}
\nonumber
&\int_\mTx \vu(t,\bx) \ccdot \vv(t,\bx) \td^3\bx - \int_\mTx \mr \vu(\bx) \ccdot \vv(0,\bx) \td^3\bx -\int_0^t \int_{\mTx} \vu\ccdot\pt\vv \td^3\bx \td s \\
\nonumber
&=\int_0^t \int_{\mTx}\big( \vu\otimes\vu- \vB\otimes\vB\big)\! : \!\nabla_{\bx}\vv \td^3\bx \td s- \int_0^t \int_\mTx \nabla_{\bx}\vu \!: \!\nabla_{\bx}\vv \td^3\bx \td s
\label{u:wk}\\
&\quad+ \int_0^t \int_{\mTx} \bigg(\vu\times \vB\,\bigg(\intp{f\,} \bigg) - \vB\times\intp{\bv f\,}  \bigg)\ccdot\vv \td^3\bx \td s\\
\intertext{for any  $\mR^3$-valued test function $\vv\in\bC^1([-1,T+1]\times\mTx;\mathbb{R}^3)$ with $\nabla_{\bx}\ccdot \vv=0$; and }
\nonumber
&\int_\mTx \vB(t,\bx) \ccdot \vv(t,\bx)\td^3\bx - \int_\mTx \mr\vB (\bx) \ccdot \vv(0,\bx) \td^3\bx -\int_0^t \int_\mTx \vB \ccdot\pt\vv \td^3\bx \td s\\
\label{B:wk}
&=\int_0^t \int_\mTx \big( \vu\otimes\vB- \vB\otimes\vu\big)\! : \!\nabla_{\bx}\vv \td^3\bx \td s- \int_0^t \int_\mTx \nabla_{\bx}\vB\! :\! \nabla_{\bx}\vv \td^3\bx \td s
\end{align}
\end{subequations}
for any  $\mR^3$-valued test function $\vv\in\bC^1([-1,T+1]\times\mTx;\mathbb{R}^3)$ with $\nabla_{\bx}\ccdot \vv=0$.
Moreover, for almost every $t\in[0,T]$, including $t=0$,
\begin{align}
\label{wk:f:Lr:pos:def}
f(t,\ccdot,\ccdot)\ge 0,\qquad \|\fb(t,\ccdot,\ccdot)\|_{ L^r(\mTx \times \mRp))}\le\|\mr f\|_{ L^r(\mTx \times \mRp)} &\quad\text{ for all $r\in[1,\infty]$},\\
\label{energy:inequality}
{1\over2}\|( \vu(t),\vB(t))\|^2_{L^2(\mTx)} +\Epar {f}(t)+ \|(\nabla_{\bx}\vu,\nabla_{\bx}\vB)\|^2_{L^2((0,t); L^{2}(\mTx))}&\le {1\over2}\| (\mr \vu,\mr\vB) \|^2_{L^2(\mTx)} +\Epar {  \mr   f}.
\end{align}
By setting $t=0$ in \eqref{wk}, we have $(f,\vu,\vB)\big|_{t=0} = (\mr f,\mr\vu,\mr\vB)$  by the du Bois-Reymond lemma, ensuring that the initial conditions at $t=0$ are satisfied.
\end{definition}

\begin{remark}
We emphasize that \eqref{wk} is valid for \emph{every} $t\in[0,T]$, although one normally sees  ``almost every $t$''  in the literature. The rationale seems to be lack of uniform-in-time convergence and the lack of (strong) compactness for $\fe$.  We will remedy this  by using a version of the Lebesgue differentiation theorem in time, and then redefining the solution at the exceptional times, which form of a set of zero Lebesgue measure in $(0,T)$, using the weak formulation.  The only adverse effect of such a redefinition is that the everywhere nonnegativity of $f$ in \eqref{wk:f:Lr:pos} is weakened to $f(t,\ccdot,\ccdot)\ge 0$  for \textit{almost} every $t\in[0,T]$.
\end{remark}

\begin{remark}\label{rm:everyT:cont}
Then, the validity of \eqref{wk} for every $t \in [0,T]$ also implies that any weak solution $(f,\vu,\vB)$ is right-continuous at $t=0$ when regarded as a continuous linear functional over the space $\bC^1_c(\mTx \times \mRp)\times \bC^1(\mTx;\mR^3) \times \bC^1(\mTx;\mR^3)$.
In fact, more is true: as the expressions appearing on the right-hand sides of \eqref{f:wk}, \eqref{u:wk} and \eqref{B:wk}
are absolutely continuous functions of $t \in [0,T]$, the same is true of the expressions on their left-hand sides, for any admissible choice of the test functions $g$ and $\vv$. By considering in particular admissible test functions $g$ and $\bV$ such that $g(t,\bx,\bv)=g_1(t)\,g_2(\bx,\bv)$ where $g_1(t)\equiv1$ for all $t \in [0,T]$, and
$\bV(t,\bx)=v_1(t)\,\bV_2(\bx)$ such that $v_1(t)\equiv1$ for all $t \in [0,T]$, we deduce that
\begin{align*}
t\in [0,T] &\mapsto \int_{\mTx \times \mRp} f(t,\bx,\bv) g_2(\bx,\bv) \td^3\bx \td^3\bv,\\
t\in [0,T]  &\mapsto \int_\mTx \vu(t,\bx)\cdot \bV_2(\bx) \td^3\bx,\\
t\in [0,T]  &\mapsto \int_\mTx \vB(t,\bx)\cdot \bV_2(\bx) \td^3\bx
\end{align*}
are absolutely continuous for any scalar-valued $g_2 \in C^1_c(\mTx \times \mRp)$,
and any $\mR^3$-valued $\bV_2 \in C^1(\mTx)$ satisfying $\nabla_{\bx} \cdot \bV_2=0$.
\end{remark}
\begin{remark}\label{rm:aeT}
Similarly as in the case of the  three-dimensional incompressible Navier--Stokes equations, it is unclear whether the energy inequality \eqref{energy:inequality} can be an equality, and whether it holds for every, rather than almost every, $t \in [0,T]$.
\end{remark}

We are now ready to prove our main result: the existence of large-data finite-energy global weak solutions to the hybrid Vlasov-MHD system.

\begin{theorem}\label{thm:incomp}
For any  $(\mr\vu,\mr\vB)\in \incL(\mTx;\mR^6)$, both of which are divergence-free in the sense of distributions, any pointwise nonnegative $ \mr f \in L^\infty(\mTx\times\mRp)\cap L^1(\mTx\times\mRp)$ with finite $\Epar{\mr f}$, and any $T>0$, there exists a weak solution $( f,\vu, \vB)$ to \eqref{eq} in the sense of Definition \ref{def:wk:inc}.
\end{theorem}
\begin{proof}

We have already found a sequence $\{(\fen,\uen,\Ben)\}_{n \geq 1}$ satisfying the assertions of Lemmas \ref{lm:comp:uB} and \ref{lm:comp:f}. This leads to the regularity and integrability properties of the limit $(f,\vu,\vB)$, as required in \eqref{def:wk:reg:int}.

The divergence-free property of $(\vu,\vB)$ required by Definition \ref{def:wk:inc}  is the consequence of the weak convergence in $L^2((0,T);H^1(\mTx;\mR^6))$, as shown in Lemma \ref{lm:comp:uB} (so that $(0,0)=(\nabla_{\bx}\ccdot\uen,\nabla_{\bx}\ccdot\Ben)$ converges weakly to $(\nabla_{\bx}\ccdot\vu,\nabla_{\bx}\ccdot\vB) $ in $L^2((0,T);L^2(\mTx;\mR^6))$ as well).

In the rest of the proof, the strong convergence properties of mollifiers asserted in Lemma \ref{lem:mo} are implicitly  used without specific referencing.

We consider any compactly supported scalar-valued function $g\in\bC^1_c([-1,T+1]\times\mTx\times\mRp)$  and   any $\mR^3$-valued  function $\vv\in\bC^1([-1,T+1]\times\mTx)$ with
$\nabla_{\bx}\cdot \vv=0$.
We shall now proceed to confirm that as $n \rightarrow \infty$ (and therefore $\ep_n\to0$) the limit of each term in the mollified weak formulation
\eqref{mn:wk} is equal to its counterpart in the weak formulation \eqref{wk}.
\begin{itemize}
\item
To prove the convergence towards the right-hand side  of \eqref{f:wk} we proceed as follows. Thanks to the strong convergence result in the function space   $L^\infty((0,T);L^5(\mTx;\mR^6))$ stated in Lemma \ref{lm:comp:uB}, we have
\[ \ueen\times\Been\to\vu\times\vB  \quad\text{strongly in $L^1((0,T);L^{5\over2}(\mTx;\mR^3))$}.\]
We then recall that $g$ is, by hypothesis, a smooth function and has compact support; therefore, both
$\int_{\mRp}|\nabla_{\bv} g|\td^3\bv$ and $\int_{\mRp}|\bv\times \nabla_{\bv} g|\td^3\bv$ are uniformly bounded
in $[0,T]\times \mTx$ as functions of $t$ and $\bx$. Thus we have that, as $n\to\infty$,
\[
\big(( \ueen -\bv)\times \Been \big)\cnp g\to \big(( \vu -\bv)\times \vB\big)\cnp g\;\;\text{ strongly in }L^1((0,T)\times \mTx \times \mRp),
\]
which, together with  the weak* convergence result \eqref{wk:conv:Linf} for $\fen$, implies that
\[
\lim_{n\to\infty}\text{(the right-hand side of \eqref{f:mn:wk})}= \text{(the right-hand side of \eqref{f:wk})}.
\]
\item
To prove the convergence towards the last integral on the right-hand side of  \eqref{u:wk}, we first rewrite the last triple product term appearing in \eqref{u:mn:wk} as  $\big( \vv^{\langle\ep \rangle}\times \Bee   \big) \ccdot \displaystyle\intp{\bv\fe}$; then by the strong convergence result stated in Lemma \ref{lm:comp:uB}, we have that, as $n\to\infty$,
\[
\begin{aligned}
\vv^{\langle\ep \rangle}\times\Been&\to\vv\times \vB &&\text{strongly in $L^2((0,T);L^{5}(\mTx;\mR^3))$, \;\; and}\\
(\uen\times\Been)\ccdot \vv&\to(\vu\times\vB)\ccdot \vv &&\text{strongly in $L^1((0,T);L^{5\over2}(\mTx))$}.
\end{aligned}
\]
Consequently, by the weak* convergence results \eqref{conv:mo:pf}, \eqref{conv:mo:f} for the moments of $\fen$ and noting that  H\"older's inequality can be applied to the difference between the last integrals in \eqref{u:mn:wk} and \eqref{u:wk}
(note that  ${1\over5}+{4\over5}={1\over5}+{1\over5}+{3\over5}=1$), we have that
\[
\lim_{n\to\infty}\text{(last integral of \eqref{u:mn:wk})}= \text{(last integral of \eqref{u:wk})}.
\]
\item The rest of right-hand side of \eqref{u:wk} and the entire right-hand side of \eqref{B:wk} involve terms which are
of one of the following two types:
\begin{itemize}
\item[1.] For terms involving the operation $\otimes$, the convergence of their counterparts in \eqref{mn:wk} follows from the strong $L^\infty((0,T);L^5(\mTx;\mR^6))$ convergence stated in Lemma \ref{lm:comp:uB}.
\item[2.] For the terms including $(\nabla_{\bx}\vu \!:\! \nabla_{\bx}\vv)$ and $(\nabla_{\bx}\vB\! : \!\nabla_{\bx}\vv)$, the convergence of their counterparts in \eqref{mn:wk} follows from the weak convergence in  $L^2((0,T);H^1(\mTx;\mR^6))$, as  stated in Lemma \ref{lm:comp:uB}.
\end{itemize}
\item To prove the convergence towards the terms appearing on the left-hand side of \eqref{wk}, we proceed by combining the strong convergence of the mollified initial data, as in \eqref{conv:mr} and Lemmas \ref{lm:comp:uB} and \ref{lm:comp:f}. We then obtain \emph{all terms on the left-hand side of \eqref{wk}, except the first,} as limits of their counterparts in \eqref{mn:wk} (recall  the hypothesis that $g$ is compactly supported).
\end{itemize}
To summarize our conclusions so far, we have
\be
\label{conv:except}
\begin{aligned}
&\text{for any }\;\;t\in[0,T]\;\;  \text{ and test functions }  g,v  \text{ as specified above,}\\
&\lim_{n\to\infty}\text{(\eqref{f:mn:wk},\,\eqref{u:mn:wk},\,\eqref{B:mn:wk})}= \text{(\eqref{f:wk},\,\eqref{u:wk},\,\eqref{B:wk}),
except the first term in each equation.}
\end{aligned}
\ee
Thus we shall now focus on the first term in each of \eqref{f:mn:wk}, \eqref{u:mn:wk}, \eqref{B:mn:wk}.
First, thanks to the strong $L^2((0,T); L^5(\mTx;\mR^6))$ convergence stated in Lemma \ref{lm:comp:uB}, there exists a subset ${\mathcal Z}_{(\vu,\vB)}\subset[0,T]$
of zero Lebesgue measure  and a subsequence of $\{\ep_n\}_{n \geq 1}$ (not indicated), such that
\begin{align}
\label{L2:ae:conv}
&\text{for any }\;\;t\in[0,T]\backslash {\mathcal Z}_{(\vu,\vB)}, \;\;  \lim_{ n\to\infty}\big\|(\uen(t),\Ben(t))-(\vu(t), \vB(t))\big\|_{L^5(\mTx)}=0,\;\;\text{ and hence,}\\
\label{uB:ae}
&\text{for   $t \in [0,T]\backslash {\mathcal Z}_{(\vu,\vB)}$,  }\lim_{n\to\infty}\text{ (first terms of \eqref{u:mn:wk}, \eqref{B:mn:wk})}= \text{ (first terms of   \eqref{u:wk}, \eqref{B:wk})}.
\end{align}
It then remains to consider the first term in \eqref{f:wk}. In the absence of a strong convergence result for the sequence $\{\fen\}_{n \geq 1}$ the argument in this case is more delicate.

To this end, we take any $\tau \in [0,T)$ and an arbitrarily small $\delta \in (0,T-\tau)$, and consider the integral average of  \eqref{f:mn:wk} over $[\tau,\tau+\delta]$:
\begin{align}
\begin{aligned}\label{time:ave-prior}
&\phantom{ =}{1\over \delta}\dint_{\tau}^{\tau+\delta}\Big[ \int_{\mTx \times \mRp} \big(\fen   \gc \big)(t,\bx,\bv) \td^3\bx \td^3\bv \Big]\td t- \int_{\mTx \times \mRp}
\mr{f}_{\ep_n}(\bx,\bv)\,   \gc(0,\bx,\bv) \td^3\bx \td^3\bv\\
&\qquad ={1\over \delta}\dint_{\tau}^{\tau+\delta} \Big\{\int_0^t \int_{\mTx \times \mRp}\fen\pt g+(\bv  \fen)\cdot \nabla_{\bx} \gc-\big(( \ueen -\bv)\times \Been \fen\big)\cnp \gc \td^3\bx \td^3\bv \td s\Big\} \td t.
\end{aligned}
\end{align}
The first term on the left-hand side here involves integration with respect the $t,\bx,\bv$ variables, so the weak* convergence \eqref{wk:conv:Linf} applies and we thus obtain the desired limit as $n \rightarrow \infty$ (and therefore as $\ep_n\to0$). The second term on the left-hand side also converges to the appropriate limit involving the given initial datum $\mr f$ for $f$. Concerning the right-hand side, the expression in the curly brackets converges pointwise, for every $t\in[\tau,\tau+\delta]$,  as stated in \eqref{conv:except};  it is dominated by a constant that is independent of $\ep_n$, thanks to the uniform energy bound \eqref{uni:bound:mn} and the fact that $g$ is smooth and compactly supported. As the constant function is, trivially, integrable over $[\tau,\tau+\delta]$, by Lebesgue's dominated convergence theorem the right-hand side above also has the desired limit as $n \rightarrow \infty$ (and $\ep_n\to0$). In short, the limit of \eqref{time:ave-prior} is
\be
\label{time:ave}
\begin{aligned}
&{1\over \delta}\dint_{\tau}^{\tau+\delta}\Big[ \int_{\mTx \times \mRp} (f\,\gc)(t,\bx,\bv)\td^3\bx \td^3\bv \Big]\td t- \int_{\mTx \times \mRp} \mr f(\bx,\bv)\,\gc(0,\bx,\bv) \td^3\bx \td^3\bv  \\
&\qquad  ={1\over \delta}\dint_{\tau}^{\tau+\delta}\Big\{\int_0^t \int_{\mTx \times \mRp} f\pt g+(\bv  f )\cdot \nabla_{\bx} \gc-\big(( \vu -\bv)\times \vB f \big)\cnp \gc\td^3\bx \td^3\bv \td s \Big\}\td t\\
&\qquad =:{1\over \delta}\dint_{\tau}^{\tau+\delta} \Big\{\int_0^t F(s)\td s\Big\}.
\end{aligned}
\ee

At the start of the proof we showed the integrability properties of the limit $(f,\vu,\vB)$, as required in \eqref{def:wk:reg:int} of  Definition \ref{def:wk:inc}. The newly defined function $F(s)$ is therefore in $L^1(0,T)$, which immediately implies the  absolute continuity of the mapping  $t \in [0,T] \mapsto\int_0^t F(s)\td s$ appearing in the right-hand side of the above equation. Thus,
\[\lim_{\delta\to0}{1\over \delta}\dint_{\tau}^{\tau+\delta} \Big\{\int_0^t F(s)\td s\Big\}=\int_0^\tau F(s)\td s\;\;\text{ for every }\;\tau\in[0,T).\]
To pass to the limit $\delta \rightarrow 0$ in the first term on the left-hand side of \eqref{time:ave}, we take a countable dense   subset $\{g_n\}_{n \geq 1}$ of  $\bC^1_c([-1,T+1]\times\mTx\times\mRp)$ consisting of compactly supported, scalar-valued test functions. Note
that we only need this countable subset to be  dense with respect to the $\bC([-1,T+1]\times\mTx\times\mRp)$ norm;
the existence of such a sequence $\{g_n\}_{n \geq 1}$ follows from Proposition \ref{separ-append1}.

Then, for a fixed $n$,  by the integrability of $f\gc_n$ and Lebesgue's differentiation theorem,
\[
\lim_{\delta\to0}{1\over \delta}\dint_{\tau}^{\tau +\delta}\Big[ \int_{\mTx \times \mRp}  (f    \gc_n )(t,\bx,\bv) \td^3\bx \td^3\bv \Big]\td t =\int_{\mTx \times \mRp}   (f    \gc_n )(\tau,\bx,\bv)\;\;\text{ for a.e. }\;\tau \in [0,T).
\]
Since the countable union of sets of zero Lebesgue measure is a set of zero Lebesgue measure, we deduce from the countability of $\{\gc_n\}_{n \geq 1}$ the existence of a set ${\mathcal Z}_{f}\subset[0,T)$ of zero Lebesgue measure such that
\be
\label{delta:gn}
\lim_{\delta\to0}{1\over \delta}\dint_{\tau }^{\tau +\delta}\Big[ \int_{\mTx\times\mRp} (f\gc_n \big)(t,\bx,\bv)\td^3\bx \td^3\bv  \Big]\td t=\int_{\mTx \times \mRp}   \big(f\gc_n \big)(\tau,\bx,\bv)\td^3\bx \td^3\bv
\;\;\;\;\forall\, n \in \mathbb{N},\;\; \forall\, \tau\in[0,T)\backslash {\mathcal Z}_{f}.
\ee
Recall that $\{\gc_n\}_{n \geq 1}$ is a dense (with respect to the $\bC$ norm) subset of $\bC^1_c([-1,T+1]\times\mTx\times\mRp)$. Consequently, for each fixed $g$ in the latter space, we can extract a subsequence  $\{\gc_{n_j}\}_{j \geq 1}$ from
$\{g_n\}_{n \geq 1}$ such that $$\|g_{n_j}-g\|_{\bC([-1,T+1]\times\mTx\times\mRp)}\le {1\over j}.$$
This implies that,  for any positive integer $j$ and any $\tau \in [0,T)\setminus {\mathcal Z}_f$,
\[
\begin{aligned}
&\left|{1\over \delta}\dint_{\tau}^{\tau +\delta}\Big[ \int_{\mTx \times \mRp} (f g)(t,\bx,\bv) \td^3\bx \td^3\bv\Big]\td t-\int_{\mTx \times \mRp} (f g)(\tau,\bx,\bv)\td^3\bx \td^3\bv \right|\\
&\quad \leq
\left|{1\over \delta}\dint_{\tau}^{\tau +\delta}\Big[ \int_{\mTx \times \mRp} (f \gc_{n_j})(t,\bx,\bv) \td^3\bx \td^3\bv\Big]\td t-\int_{\mTx \times \mRp} (f \gc_{n_j})(\tau,\bx,\bv)\td^3\bx \td^3\bv \right| +{2\over j}\|f\|_{L^\infty((0,T);L^1(\mTx \times \mRp))}.
\end{aligned}
\]
By \eqref{delta:gn} we then have, for any positive integer $j\in \mathbb{N}$ and all $\tau\in[0,T)\backslash {\mathcal Z}_{f}$, that
\[
\limsup_{\delta\to0}\left|{1\over \delta}\dint_{\tau}^{\tau +\delta}\Big[ \int_{\mTx \times \mRp} (f g)(t,\bx,\bv) \td^3\bx \td^3\bv\Big]\td t-\int_{\mTx \times \mRp} (f g)(\tau,\bx,\bv)\td^3\bx \td^3\bv \right| \leq {2\over j}\|f\|_{L^\infty((0,T);L^1(\mTx \times \mRp))}.
\]
Passing to the limit $j \rightarrow \infty$, we deduce that
\[
\lim_{\delta\to0}{1\over \delta}\dint_{\tau}^{\tau +\delta}\Big[ \int_{\mTx \times \mRp} (f g)(t,\bx,\bv) \td^3\bx \td^3\bv\Big]\td t=\int_{\mTx \times \mRp} (f g)(\tau,\bx,\bv)\td^3\bx \td^3\bv
\quad \mbox{for all $\tau \in [0,T)\backslash {\mathcal Z}_{f}$.}
\]
Combining this with the convergence shown below  \eqref{time:ave},
we have proved that the $\delta\to0$ limit of   \eqref{time:ave} is indeed \eqref{f:wk} for every $t=\tau\in[0,T)\backslash  {\mathcal Z}_{f}$ and a fixed $g$. As the set ${\mathcal Z}_f$ that is chosen in the line above \eqref{delta:gn} is
independent of $g$, we have that
\eqref{f:wk} holds on $[0,T)\setminus {\mathcal Z}_f$ for all admissible test functions $g$.

By combining this assertion with \eqref{conv:except}, \eqref{uB:ae} we finally deduce that
\[
\text{\eqref{wk} holds for every $\tau \in[0,T)\backslash({\mathcal Z}_{(\vu,\vB)}\cup {\mathcal Z}_{f})$;}
\]
in other words, \eqref{wk} hold a.e. on $[0,T]$.

Regarding the ``exceptional times'', at every  $t\in  {\mathcal Z}_{f}$ (resp. $t\in  {\mathcal Z}_{(\vu,\vB)}$), we use \eqref{f:wk} (resp. \eqref{u:wk} and \eqref{B:wk}) to redefine $f$ (resp.  $\vu$ and $\vB$)  as an element in the dual space of $\bC^1_c( \mTx\times\mRp)$ (resp. $\bC^1 (\mTx;\mR^3)$).

In the final part of the proof, we show the physically relevant properties \eqref{wk:f:Lr:pos:def}, \eqref{energy:inequality} for a.e. $t\in[0,T]$; we note that they hold at $t=0$, since the initial conditions are precisely satisfied, which is shown immediately after the inequality \eqref{energy:inequality}.

The nonnegativity of $f$ stated in \eqref{wk:f:Lr:pos:def} follows from \eqref{wk:f:Lr:pos} and the fact that we have only redefined $f(t,\ccdot,\ccdot)$ on a subset of $[0,T]$ of zero Lebesgue measure. The estimate in   \eqref{wk:f:Lr:pos:def} is just a duplication of \eqref{fb:Lr:le}, which is unaffected by the redefinition procedure.

To show the energy inequality \eqref{energy:inequality}, we choose the test function in \eqref{f:wk} from the sequence of nonnegative, $(t,\bx)$-independent functions $\{g_k\}_{k \geq 1}\subset \bC^1_c(\mRp)$, which are nondecreasing pointwise (i.e., $0\le g_1(\bv)\le g_2(\bv)\le\cdots$) and converge to $|\bv|^2$ pointwise. Since we have established that each term but the first one in \eqref{f:wk} is the limit of its counterpart in  \eqref{f:mn:wk} as $n\to\infty$, then so is the first term in \eqref{f:wk}; thus, for a.e. $t\in[0,T]$ we have
\begin{align*}
\int_{\mTx \times \mRp} f(t,\bx,\bv)\,g_k(\bv) \td^3\bx \td^3\bv &=\lim_{n\to\infty} \int_{\mTx \times \mRp} \fen(t,\bx,\bv)\, g_k(\bv) \td^3\bx \td^3\bv\\
&\le \liminf_{n\to\infty} \int_{\mTx \times \mRp} \fen(t,\bx,\bv)\,   |\bv|^2 \td^3\bx \td^3\bv.
\end{align*}
Hence, by the monotone convergence theorem (noting the nonnegativity of $f$ in \eqref{wk:f:Lr:pos:def}), for a.e. $t\in[0,T]$,
\[
\int_{\mTx \times \mRp}   f(t,\bx,\bv)   |\bv|^2 \td^3\bx \td^3\bv \le \liminf_{n\to\infty} \int_{\mTx \times \mRp}   \fen(t,\bx,\bv)\,|\bv|^2 \td^3\bx \td^3\bv.
\]
Combining this with the mollified version of the energy equality \eqref{energy:mollify}, the bound on the initial energy \eqref{IC:en:bound}, the convergence of the fluid energy, which follows from \eqref{L2:ae:conv} and H\"older's inequality, and the weak $L^2((0,T);H^1(\mTx;\mR^6))$ convergence result from Lemma \ref{lm:comp:uB}, we deduce the energy inequality \eqref{energy:inequality} for a.e. $t\in[0,T]$.
\end{proof}

\subsection{Additional properties of weak solutions}
By relying on Definition \ref{def:wk:inc} {\it only}, it is possible to show that weak solutions possess the following additional properties.
\begin{proposition}
Let $(f,\vu,\vB)$ be a weak solution in the sense of Definition \ref{def:wk:inc}. Then, the following properties hold:
\begin{enumerate}
\item[(a)] $t \in [0,T] \mapsto f(t,\ccdot,\ccdot) \in \big(\bC^1_c( \mTx\times\mRp)\big)^*$ and $t \in [0,T] \mapsto (\vu,\vB)(t,\ccdot,\ccdot) \in \big(\bC^1(\mTx;\mR^6)\big)^*$ are absolutely continuous mappings.
\item[(b)] The zeroth and first moment of $f$ satisfy, respectively,
\be
\label{prop:f:pf}
\intp{\fb(\cdot,\cdot,\bv)\,}\in L^{\infty}((0,T); L^{5\over3}(\mTx))\quad \mbox{and}\quad
\intp{|\bv|\,\fb(\cdot,\cdot,\bv)\,}\in L^{\infty}((0,T); L^{5\over4}(\mTx)).
\ee
\item[(c)] The total momentum is conserved, i.e., for almost every $t\in[0,T]$,
\be
\label{prop:cons:mm}
\int_{\mTx \times \mRp} \bv\,f(t,\bx,\bv)\td^3\bx \td^3\bv  + \int_\mTx  \vu(t,\bx) \td^3\bx =  \int_{\mTx \times \mRp} \bv\,\mr f(\bx,\bv) \td^3\bx \td^3\bv  + \int_\mTx \mr \vu(\bx) \td^3\bx.
\ee
\item[(d)] The total mass of energetic particles is conserved, i.e., for almost every $t\in[0,T]$,
\be
\label{prop:cons:mass}
\int_{\mTx \times \mRp} f(t,\bx,\bv) \td^3\bx \td^3\bv  =  \int_{\mTx \times \mRp} \mr f(\bx,\bv) \, p \td^3\bx \td^3\bv.
\ee
\end{enumerate}
\end{proposition}
\begin{proof}
We prove the assertions item by item.
   \begin{enumerate}
\item[(a)] The stated absolute continuity properties directly follow from Remark \ref{rm:everyT:cont}.
\item[(b)] \eqref{prop:f:pf} is a consequence of Proposition \ref{prop:Lr}, Definition \ref{def:wk:inc} and the energy inequality \eqref{energy:inequality}.
\item[(c)] We choose the test function appearing in \eqref{f:wk} from the sequence of $(t,\bx)$-independent functions $\{\tilde g_k\}_{k \geq 1}\subset \bC^1_c(\mRp)$   so that $\{\tilde g_k\}_{k \geq 1}$ converges to $v_1$ (the first coordinate of $\bv=(v_1,v_2,v_3)$) pointwise, $\{\nabla_{\bv}\tilde g_k\}_{k \geq 1}$ converges to  $(1,0,0)^{\rm T}$  pointwise, and
\be
\label{tld:gk}| \tilde g_k(\bv)|\le|\bv_1|, \quad|\nabla_{\bv}\tilde g_k(\bv)|\le1 .
\ee
Then, by Definition \ref{def:wk:inc}, for a.e. $t\in[0,T]$, the measurable function $f(t,\ccdot,\ccdot)$ satisfies
\begin{align*}
\int_{\mTx \times \mRp} f(t,\bx,\bv)\,  \tilde g_k(\bv) \td^3\bx \td^3\bv - \int_{\mTx \times \mRp} \mr f(\bx,\bv)\, \tilde  g_k(\bv) \td^3\bx \td^3\bv\\
= -\int_0^t \int_{\mTx \times \mRp} \big(( \vu -\bv)\times \vB f\big)\cnp \tilde g_k \td^3\bx \td^3\bv \td s.
\end{align*}
By  \eqref{def:wk:reg:int}, \eqref{prop:f:pf},  \eqref{tld:gk},  each of the integrands is uniformly bounded by an integrable
function, so that by taking the $k\to\infty$ limit and applying Lebesgue's dominated convergence theorem we have
\begin{align*}
\qquad\qquad\int_{\mTx \times \mRp}f(t,\bx,\bv)\, v_1 \td^3\bx \td^3\bv  - \int_{\mTx \times \mRp} \mr f(\bx,\bv)\, v_1 \td^3\bx \td^3\bv \\
= -\int_0^t \int_{\mTx \times \mRp} \big(( \vu -\bv)\times \vB f\big)\cnp v_1 \td^3\bx \td^3\bv \td s.
\end{align*}
We then choose the test function $\bv=(1,0,0)^{\rm T}$ in \eqref{u:wk} and add the resulting equality to the one above to deduce, for a.e. $t\in[0,T]$, that
\[
\int_{\mTx} f(t,\bx,\bv)\, v_1 \td^3\bx \td^3\bv  +\int_\mTx U_1(t,\bx) \td^3\bx   =  \int_{\mTx \times \mRp} \mr f(\bx,\bv)\, v_1 \td^3\bx \td^3\bv + \int_\mTx \mr  \vu_1(\bx) \td^3\bx,
\]
where $U_1$ is the first component of the velocity vector $\bU=(U_1, U_2, U_3)^{\rm T}$. Likewise, we can show the conservation of the second and third components of the total momentum, hence proving \eqref{prop:cons:mm}.
\item[(d)]
To show   the conservation of the mass of energetic particles, as stated in \eqref{prop:cons:mass}, we change the sequence $\{\tilde g_k\}_{k \geq 1} \subset \bC^1_c(\mTx\times\mRp)$, used in the proof of item (c) above, so that its elements are dominated by, and converge to, $1$ pointwise, with $\nabla_{\bv}\tilde g_k$ dominated by $1$ and converging to $0$ pointwise. We skip the remaining steps, as they are an easier version of the proof of item (c) above.
\end{enumerate}
That completes the proof of the proposition.
\end{proof}

\subsection*{Acknowledgments} C.T. acknowledges financial support from the Leverhulme Trust Research Project Grant No. 2014-112, and by the London Mathematical Society Grant No. 31439 (Applied Geometric Mechanics Network). The authors thank the anonymous referees for their insightful comments and careful proofreading, which greatly helped to improve the structure and rigor of this paper.

\appendix

\section{\bf Product rules}\label{app:prod}
The following identities are useful variants of the product rule, where $f$ and $g$ are scalar-valued functions while $\vec u$ and $\vec v$ are $\mR^3$-valued functions:
 %
\begin{align}
\label{A1}
\nabla (fg) &= f\nabla g + g\nabla f;\\
\label{A2}
\nabla ({\bu} \cdot  \bv) &= \bu \times (\nabla \times \bv) + \bv \times (\nabla \times \bu) + ( \bu \cdot \nabla_{\bx}) \bv + (\bv \cdot\nabla )\bu;\\
\label{A3}
\nabla\cdot (f \bv) &= f (\nabla\cdot\bv) + \bv \cdot (\nabla f);\\
\label{A4}
\nabla\cdot (\bu \times \bv)& = \bv \cdot (\nabla \times \bu) - \bu \cdot (\nabla \times \bv);\\
\label{A5}
\nabla \times (f \bv) &= (\nabla f) \times \bv + f (\nabla \times \bv);\\
\label{A6}
\nabla \times (\bu \times \bv) &= \bu \, (\nabla \cdot \bv) - \bv \, (\nabla \cdot \bu) + (\bv \cdot \nabla) \, \bu - (\bu \cdot \nabla) \, \bv;\\
\label{A7}
(\nabla \times \bu)\times\bu&=(\bu\cdot \nabla)\bu-{1\over2}\nabla |\bu|^2.
\end{align}

\section{\bf The linear(ized) incompressible MHD system}

Let $\pair{\cdot,\cdot}$ denote the usual $L^2(\mTx)$ inner product of scalar-valued or $\mR^3$-valued   functions. We shall prove the following
result.
\begin{lemma}
\label{thm:linearMHD}
Given any $T>0$, consider the following forced, linear(ized) incompressible MHD system over the domain $[0,T]\times\mTx$:
\be
\label{fMHD}
\left\{
\begin{aligned}
\pt\vu+\ba\cdot \nabla_{\bx}\vu-\bb \cdot \nabla_{\bx}\vB- \Delta_{\bx}\vu &=\vu\times \bg+\bh -\nabla_{\bx} \mathcal{P}\quad &&\quad\textnormal{(subject to }\nc\vu=0), \\
\pt \vB+\ba\cdot \nabla_{\bx}\vB-\bb\cdot \nabla_{\bx}\vu- \Delta_{\bx}\vB&=\bh_1-{ \nabla_{\bx} \mathcal{P}_\vB}\quad&&\quad\textnormal{(subject to }\nc\vB=0),
\end{aligned}
\right.
\ee
\[
\textnormal{with divergence-free initial data}\quad ( \vu, \vB)\IC=  (\mr\vu,\mr\vB) \in\Cinf(\mTx;\mathbb{R}^6),
\]
where $\ba$, $\bb$, $\bg$, $\bh$, $\bh_1 \in \bigcap_{m \geq 0}\mathcal{C}([0,T];\mathcal{C}^m(\mathbb{T}^3; \mathbb{R}^3))$
are given $\mR^3$-valued functions with $\nc \ba=\nc \bb=0$.

Then, this system admits a   classical solution
\[(\vu,\vB) \in \bigcap_{m \geq 0}\mathcal{C}([0,T];\mathcal{C}^m(\mathbb{T}^3; \mathbb{R}^6)) \]
that satisfies
\be
\label{energy:fMHD}
\|(\vu,\vB) \|^2_{L^2(\mTx)}\Big|_0^T +2\int_0^T  \|(\nabla_{\bx}\vu,\nabla_{\bx}\vB) \|^2_{L^2(\mTx)} \td t=
 2\int_0^T\pair{\vu ,\bh}+\pair{\vB ,\bh_1}\td t,
\ee
\begin{align}
\|(\nabla_{\bx}\vu,\nabla_{\bx}\vB) \|^2_{L^2(\mTx)}\Big|_0^T &+ \int_0^T\!  \|(\pt\vu,\pt\vB) \|^2_{L^2(\mTx)} \td t\nonumber\\
\label{energy:pt}
&\le2\sup_{[0,T]\times\mTx}\{|\ba|^2,|\bb|^2,|\bg|^2,1\}\int_0^T \!\|(\nabla_{\bx}\vu,\nabla_{\bx}\vB,\vu,\bh,\bh_1)\|_{L^2(\mTx)}^2\td t.
\end{align}
\end{lemma}
\begin{proof}
Let $\{\bW_j\}_{j\geq 1}\subset\Cinf(\mTx;\mR^3)$ be an orthonormal basis of eigenfunctions of $\Delta_{\bx}$ in the space of divergence-free $\mathbb{R}^3$-valued vector fields in $L^2(\mTx;\mathbb{R}^3)$.
For any, not necessarily divergence-free, vector field $\vv\in \bC(\mTx;\mR^3)$ and an integer $N \geq 1$, we define the projection
\[
\scP_N[\vv]:=\sum_{j=1}^N  \pair{  \vv,\bW_j }\bW_j.
\]
Further, we define an approximation  of $(\vu,\vB)$ by
\[
\bpm\vu^\N(t,\bx)\\\vB^\N(t,\bx)\epm=\sum_{j=1}^N\bpm \fru_j(t)\\\frB_j(t)\epm\bW_j(\bx),\qquad t \in [0,T], \;\; \bx \in \mTx,
\]
that satisfies, for $i=1,2,\dots,N$,
\[
\left\{
\begin{aligned}
\ddt \fru_i(t)&=\sum_{j=1}^N  \left[\pair{\bW_i,\ba\cdot \nabla_{\bx} \bW_j }\fru_j-\pair{\bW_i,\bb\cdot \nabla_{\bx} \bW_j}\frB_j+\pair{\bW_i,\Delta_{\bx} \bW_j }\fru_j+\pair{\bW_i,\bW_j\times \bg }\fru_j \right]+\pair{\bW_i,  \bh }, \\
\ddt \frB_i(t)&=\sum_{j=1}^N  \left[\pair{\bW_i,\ba\cdot \nabla_{\bx} \bW_j }\frB_j-\pair{\bW_i,\bb\cdot \nabla_{\bx} \bW_j}\fru_j+ \pair{\bW_i,\Delta_{\bx} \bW_j }\frB_j\right]+\pair{\bW_i,  \bh_1 },
\end{aligned}
\right.
\]
\[
\textnormal{with initial data}\quad ( \vu^\N, \vB^\N)\IC =\scP_N (\mr\vu,\mr\vB).
\]
This is a closed, $2N$-by-$2N$ system of linear ordinary differential equations for the unknowns $\{ \fru_j(t),\frB_j(t)\}_{j=1}^N$ with coefficients depending on $\ba, \bb, \bg, \bh$, $\bh_1$, which are in $\bC([0,T]; \bC^m(\mTx;\mR^3))$ for any $m\ge0$. Therefore, it admits a  solution satisfying
\be
\label{uB:N:Cm}
(\vu^\N,\vB^\N)\in\bC([0,T]; \bC^m(\mTx;\mR^6)),\quad \pt(\vu^\N,\vB^\N)\in\bC([0,T]; \bC^m(\mTx;\mR^6)),\quad
\text{ for all }m\ge0.
\ee
We note that the above system of linear ordinary differential equations is equivalent to
\be
\label{uB:cPN}
\left\{
\begin{aligned}
\pt\vu^\N&= \scP_N\left[-\ba\cdot \nabla_{\bx}\vu^\N+\bb\cdot \nabla_{\bx}\vB^\N+ \Delta_{\bx}\vu^\N+\vu^\N\times \bg+  \bh\right],\\
\pt \vB^\N&=\scP_N\left[  -\ba\cdot \nabla_{\bx}\vB^\N+\bb\cdot \nabla_{\bx}\vu^\N+ \Delta_{\bx}\vB^\N+\bh_1\right].
\end{aligned}
\right.
\ee

The regularity of  $(\vu^\N,\vB^\N) $ in \eqref{uB:N:Cm} and the $\bW_j$'s being eigenfunctions of $\Delta_{\bx}$ allow us to apply the energy method to \eqref{uB:cPN} and integrate by parts (noting that $\nc \ba=\nc \bb=0$) to obtain
\[
 \ddt\big(\|\vu^\N\|^2_{L^2(\mTx)}+\|\vB^\N\|^2_{L^2(\mTx)}\big)=
 -2\big(\|\nabla_{\bx}\vu^\N\|^2_{L^2(\mTx)}+\|\nabla_{\bx}\vB^\N\|^2_{L^2(\mTx)}\big)+2\pair{\vu^\N,\bh}+2\pair{\vB^\N,\bh_1}.
\]
We then apply the energy method to the action of the differential operator $\Delta_{\bx}^m$ on \eqref{uB:cPN},
noting that spatial differential operators commute with the projector $\mathcal{P}_N$ thanks to the periodic boundary conditions with respect to $\bx$, and we integrate by parts and apply the Cauchy--Schwarz inequality to obtain
\[
\begin{aligned}
\ddt\bigg(\|\Delta_{\bx}^m\vu^\N\|^2_{L^2(\mTx)}+\|\Delta_{\bx}^m\vB^\N\|^2_{L^2(\mTx)}\bigg)&\le  C\|(\ba,\bb)\|_{H^m(\mTx)}\|(\vu^\N,\vB^\N)\|_{H^m(\mTx)}\\
&\quad+C\|\vu^\N\|_{H^m(\mTx)}\|\bh\|_{H^m(\mTx)}+C\|\vB^\N\|_{H^m(\mTx)}\|\bh_1\|_{H^m(\mTx)},
\end{aligned}
\]
where  $m\ge 4$ and also $C$ denotes a generic constant that may depend on $m$ but is independent of $N$.
Since $\ba,\bb \in \mathcal{C}^m(\mTx;\mathbb{R}^3)$ and $(\vu,\vB)\big|_{t=0} \in \mathcal{C}^m(\mTx;\mathbb{R}^6)$
for all $m \geq 0$, we can
integrate the sum of the above equality and inequality in time to obtain
\[
\big\|(\vu^\N,\vB^\N)\big\|_{ \bC([0,T],H^m(\mTx;\mR^6))} \leq F_1\Big(T,(\vu,\vB)\big|_{t=0},\ba,\bb,\bg,\bh,\bh_1,m\Big)\qquad \forall\,N \in \mathbb{N}.
\]
Here and below, $F_1,F_2$ are functions that are independent of $N$.  Also, take the $L^2(\mTx;\mR^3)$ inner product of $\pt \vu^\N$ and $\pt \vB^\N$ with the respective equations of \eqref{uB:cPN}, apply the Cauchy--Schwarz
inequality and combine the resulting bound with the previous estimate to deduce that
\[
\big\|\pt(\vu^\N,\vB^\N)\big\|_{ \bC([0,T],H^{m-2}(\mTx))} \leq F_2\Big(T,(\vu,\vB)\big|_{t=0},\ba,\bb,\bg,\bh,\bh_1,m\Big)\qquad \forall\,N \in \mathbb{N}.
\]
Thus, these two  uniform-in-$N$ estimates imply that   the family $\{(\vu^\N,\vB^\N)\}_{N\geq 1}$, viewed as a sequence of continuous mappings from
$[0, T ]$ into $H^{m-2}(\mTx;\mR^6)$, is equicontinuous. Therefore, by the Arzel\`a--Ascoli theorem,
there exists a pair
\[
(\vu,\vB)\in \bC([0, T ];\bC^{m-4}(\mTx;\mR^6))
\]
such that upon subtracting a subsequence and using the embedding $H^{m-2}(\mTx;\mR^3) \subset \bC^{m-4}(\mTx;\mR^3)$,
\be
\label{limit:uB}
\lim_{N\to\infty}( \vu^\N, \vB^\N) = (\vu,\vB) \;\;\text{ strongly and uniformly in }  \bC([0, T ];\bC^{m-4}(\mTx;\mR^6)).
\ee
By recalling that $\{\bW_j\}_{j \geq 1}$ are, by definition, divergence-free, we deduce that the above limit $(\vu,\vB)$ also satisfies the divergence-free condition for all $t\in[0,T]$. Moreover, thanks to the choice of initial data $( \vu^\N, \vB^\N)\IC = \scP_N (\mr\vu,\mr\vB)$ and the completeness of the basis $\{\bW_j\}_{j \geq 1}$, we have  $( \vu , \vB )\IC = (\mr\vu,\mr\vB)$.

In order to show that $(\vu,\vB)$, found in this way, does indeed satisfy the equation \eqref{fMHD}, we only take $\vB$ for example.
We integrate the second equation of \eqref{uB:cPN} with respect to $t$ from 0 to $\tau\le T$,  apply the projection $\scP_{N_1}$ and use $\scP_{N_1}\scP_N=\scP_{\min\{N,N_1\}}$ to deduce that
\[
\scP_{N_1} \vB \Big|_0^{\tau} =\int_0^{\tau}\scP_{N_1}  [  -\ba\cdot \nabla_{\bx}\vB^\N+\bb\cdot \nabla_{\bx}\vu^\N+ \Delta_{\bx}\vB^\N +\bh_1]\td t\quad\text{ for any }N\ge N_1.
\]
%
By holding $N_1 \geq 1$ fixed and letting  $N\to\infty$,  we have (thanks to  \eqref{limit:uB}),
\[
\scP_{N_1}\vB \Big|_0^{\tau} = \int_0^{\tau} \scP_{N_1} [-\ba\cdot \nabla_{\bx}\vB +\bb\cdot \nabla_{\bx}\vu + \Delta_{\bx}\vB+\bh_1] \td t
\]
for any $N_1\geq1$ and any $\tau \in [0,T]$, and therefore $\vB$ is a solution to the second equation in \eqref{fMHD}.

The energy equality \eqref{energy:fMHD} is simply a consequence of taking  the $L^2(\mTx;\mR^3)$ inner product of the first and second equation in \eqref{fMHD} with $\vu$ and $\vB$, respectively, adding up, performing integrations by parts (thanks to the spatial regularity of every term), and cancellation using $\nc\vu=\nc\vB=0$.

Finally, for \eqref{energy:pt}, we use the first equation (which is the more difficult one) of \eqref{fMHD} as an example. The regularity of each term allows us to take its  $L^2(\mTx;\mR^3)$ inner product with $\pt\vu$ and perform integrations by parts to arrive at
\[
\begin{aligned}
\|\pt\vu\|_{L^2(\mTx)}^2+{1\over2}\ddt\|\nabla_{\bx}\vu\|_{L^2(\mTx)}^2&=\pair{\pt\vu,-\ba\cdot \nabla_{\bx}\vu+\bb\cdot \nabla_{\bx}\vB +\vu\times \bg+\bh }\\
&\le\|\pt\vu\|_{L^2(\mTx)}\|(\nabla_{\bx}\vu,\nabla_{\bx}\vB,\vu,\bh)\|_{L^2(\mTx)}\,\sup\{|\ba|,|\bb|,|\bg|,1\}\\
&\le{1\over2}\|\pt\vu\|_{L^2(\mTx)}^2+{1\over2}\|(\nabla_{\bx}\vu,\nabla_{\bx}\vB,\vu,\bh)\|_{L^2(\mTx)}^2\,\sup\{|\ba|^2,|\bb|^2,
|\bg|^2,1\}.
\end{aligned}
\]
We treat the second equation of \eqref{fMHD} similarly and add the resulting inequality to the one above, and finally integrate the resulting sum in time to deduce \eqref{energy:pt}.
\end{proof}

\section{\bf Separability of $\bC_c^1([T_1,T_2] \times \mTx\times\mRp)$}\label{app:C}

For $T_1 < T_2$, denote by $\bC_c^1([T_1,T_2] \times \mTx\times\mRp)$ the set of all real-valued, compactly supported functions in $\bC^1([T_1,T_2] \times \mathbb{R}^3 \times \mathbb{R}^3)$, which are $2\pi$-periodic with respect to their second argument, $\bx$, for all $(t,\bv) \in [T_1,T_2] \times \mathbb{R}^3$. Our goal is to show that $\bC_c^1([T_1,T_2] \times \mTx\times\mRp)$ is separable with respect to the $\bC$ norm.  We shall rely on the classical Stone--Weierstrass theorem.

\begin{theorem}[Stone--Weierstrass theorem]
Suppose that $\mU$ is a compact Hausdorff space and $\sA$ is a subalgebra of the space of real-valued continuous functions $\bC(\mU; \mR)$, which contains a nonzero constant function. Then $\sA$ is dense in $\bC(\mU; \mR)$ if, and only if, it separates points.
\end{theorem}

A set $\sA\subset \bC(\mU; \mR)$ is said to separate points in $\mU$ if, for any $u\ne u'\in \mU$, there exists at least one element $g\in \sA$ such that $g(u)\ne g(u')$. The set $\sA$ may or may not be countable.

Now, for any positive integer $N$, let ${{\mathsf B}_{\mRp}(\mathbf{0},N)}$ denote the {\it open} ball in $\mRp$ centered at the origin and of radius $N$, and define  $\mU_N:=(\mTx\times {{\mathsf B}_{\mRp}(\mathbf{0},N)})\subset(\mTx\times\mRp)$. We then have the following natural consequence of the Stone--Weierstrass theorem.
\begin{proposition}\label{separ-append-compact}Let $\bC(\overline{\mU_N})$   denote the space of all continuous functions that are {\it defined} on the compact domain $\overline{\mU_N}$. Then, $\bC(\overline{\mU_N})$  is separable with respect to  the $\bC$ norm.
\end{proposition}

\begin{proof}
Our proof consists of two steps.

\textbf{\textit{Step 1.}} In this step we construct an uncountable dense subset $\sA$ of $\bC(\overline{\mU_N})$. Clearly, $\bC(\overline{\mU_N})$  is an algebra over the field $\mR$. By the Stone--Weierstrass theorem, we need at least one nonzero constant function $1\in\sA$, and some other elements to ``separate'' point-pairs in $\mU_N$. To this end, consider two different points $u=(x_1,x_2,x_3,v_1,v_2,v_3)$, $u'= (x_1',x_2',x_3',v_1',v_2',v_3')$. If     $v_1\ne v_1'$, then  obviously $g(u)=v_1$ ``separates''  $u$ and $u'$. So we include $v_1\in\sA$, and likewise for $v_2$ and $v_3 $. If  $x_1\ne x_1'$ as elements of $\mT=\mR/(2\pi\mZ)$, then $\sin(x_1)$ and $\cos(x_1)$ together  ``separate''  $u$ and $u'$. Otherwise, having both \(\sin(x_1)=\sin(x_1')$ and   $\cos(x_1)=\cos(x_1')\) would imply that
\[
0=(\sin(x_1)-\sin(x_1'))^2+(\cos(x_1)-\cos(x_1'))^2=2-2\cos(x_1-x_1'),
\]
which would mean that $x_1$ and $x_1'$ are identical elements of $\mT$, thus contradicting the assumption that $x_1 \neq x_1'$ as elements of $\mT$. Thus, we include $\sin(x_1)$ and $\cos(x_1)$ in $\sA$, and likewise for the $x_2$ and $x_3$ coordinates.

In summary, the smallest subalgebra $\sA$ of $\bC(\overline{\mU_N})$  that contains
\[
\{1,v_1,v_2,v_3,\sin(x_1),\cos(x_1),\sin(x_2),\cos(x_2),\sin(x_3),\cos(x_3)\}
\]
is dense in $\bC(\overline{\mU_N})$. However, this algebra is over the field $\mR$, and therefore it is not countable.

{\textbf{\textit{Step 2.}}} Let $\sA_{\mQ}$ therefore be the smallest algebra over the field $\mQ$ of rational numbers that contains the above 10 functions. Since $\mQ$ is dense in $\mR$, and the above 10 functions are clearly bounded over the {\it compact} domain $\overline{\mU_N}$, we have that $\sA_{\mQ}$ is dense in $\sA$ and thus also dense in $\bC(\overline{\mU_N})$ . Therefore, we have shown that $\bC(\overline{\mU_N})$  is separable with respect to  the $\bC$ norm and hence the proof is complete.
\end{proof}

We also need the following proposition.
\begin{proposition}\label{separ-append-subspace} Given a separable \emph{metric} space $\bX$,
 any set $\sfS\subset \bX$ has a countable subset that is dense in $\sfS$ with respect to the $\bX$ metric.
\end{proposition}

\begin{proof}
To prove this, let $d(\ccdot,\ccdot)$ be the metric of $\bX$, and let $\{a_n\}_{n\geq 1}$ be dense in $\bX$.  Let ${\mathsf B}(a_n,{{1\over k}})$ be the open ball centered at $a_n$ of radius ${1\over k}$ in the $\bX$ metric. We then have countably many sets of the form
\[\sfS\cap {\mathsf B}\bigg(a_n,{{1\over k}}\bigg),\quad n,k\in\mathbb{N}.\]
For any nonempty such set, we pick a ``representative'' \(a_{n,k}\in \sfS\cap {\mathsf B}(a_n,{{1\over k}})\).
We claim that the countable set
\[\bigg\{a_{n,k}\in \sfS\cap {\mathsf B}\bigg(a_n,{{1\over k}}\bigg)\,\Big|\, \sfS\cap {\mathsf B}\bigg(a_n,{{1\over k}}\bigg)\ne\emptyset,\,{n,k\in\mathbb{N}}\bigg\}\subset \sfS\]
is dense in $\sfS$. Indeed, given any $k\in\mathbb{N}$ and $b\in \sfS$, by the   density of $\{a_n\}_{n \geq 1}$, we can find an $a_n$ such that
\[
d(a_n,b)<{1\over k}\implies b\in {\mathsf B}\bigg(a_n,{{1\over k}}\bigg).
\]
Hence, $b\in\sfS\cap {\mathsf B}(a_n,{{1\over k}})\ne\emptyset$, so the  ``representative'' $a_{n,k}\in\sfS\cap {\mathsf B}(a_n,{{1\over k}})$ exists. Now, since $b$ and $a_{n,k}$ are both contained in the ball ${\mathsf B}(a_n,{{1\over k}})$, it follows that    $d(a_{n,k},b)<2/k$. Thus, by the arbitrariness of $k$, we complete the proof.
\end{proof}

We are ready to state and prove the main proposition of this appendix.
\begin{proposition}\label{separ-append1}For any $T_1<T_2$,
the  space $\bC_c^1([T_1,T_2]\times\mTx\times\mRp)$ is separable with respect to  the $\bC$ norm.
\end{proposition}
\begin{proof}
By Proposition \ref{separ-append-subspace}, it suffices to prove the separability of $\bC_c^1(\mR\times\mTx\times\mRp)$. Since for functions defined on $\mR\times\mTx\times\mRp$ the proof is identical to the one in the case of functions defined on $\mTx\times\mRp$, for the sake of simplicity we shall only show that the space $\bC_c^1( \mTx\times\mRp)$ is separable with respect to the $\bC$ norm. To this end, according to Proposition \ref{separ-append-subspace} again, it suffices to show that:
\be\label{sep:x:p}\text{The  space $\bC_c( \mTx\times\mRp)$ is separable with respect to  the $\bC$ norm.}\ee

For any  positive integer $N$, the space $\bC_c(\mU_N)$ can be regarded as a subspace of $\bC(\overline{\mU_{N}})$. Then, by combining Propositions \ref{separ-append-compact} and \ref{separ-append-subspace}, we obtain that $\bC_c(\mU_N)$ has a countable  subset that is dense in $\bC_c(\mU_N)$ with respect to the $\bC$ norm.

On the other hand, any element of $\bC_c(\mU_N)$ can be naturally extended from  $\mU_N$ to the whole of $\mTx\times\mRp$ and
it can be therefore viewed as an element of $\bC_c( \mTx\times\mRp)$. Hence,
\[\bC_c( \mTx\times\mRp)= \bigcup_{N \geq 1}\bC_c(\mU_N).\]
By a countability argument and the previous step we then deduce \eqref{sep:x:p}.
\end{proof}

\end{document}